\documentclass[reqno]{amsart}
\usepackage{amsmath,amsthm,amscd,amssymb,amsfonts, amsbsy}
\usepackage{latexsym, color}
\usepackage{hyperref}
\usepackage{pxfonts}
\usepackage{mathrsfs}
\usepackage{bbm}
\usepackage{mathtools}
\usepackage{enumitem}

\theoremstyle{plain}
\newtheorem{main}{Theorem}
\newtheorem{theorem}[equation]{Theorem}
\newtheorem{lemma}[equation]{Lemma}

\newtheorem{proposition}[equation]{Proposition}

\theoremstyle{definition}
\newtheorem{definition}[equation]{Definition}
\newtheorem{condition}[equation]{Condition}

\theoremstyle{remark}
\newtheorem{remark}[equation]{Remark}

\newcommand{\capacity}{\operatorname{cap}}

\newcommand{\dist}{\operatorname{dist}}
\newcommand{\diam}{\operatorname{diam}}
\newcommand{\intr}{\operatorname{int}}

\numberwithin{equation}{section}

\newcommand{\bR}{\mathbb{R}}

\providecommand{\set}[1]{\{#1\}}
\providecommand{\Set}[1]{\left\{#1\right\}}

\providecommand{\abs}[1]{\lvert#1\rvert}
\providecommand{\Abs}[1]{\left\lvert#1\right\rvert}

\providecommand{\norm}[1]{\lVert#1\rVert}

\renewcommand{\vec}[1]{\boldsymbol{#1}}

\begin{document}
\title[Dirichlet problem for elliptic equations in non-divergence form]
{The Dirichlet problem for second-order elliptic equations in non-divergence form with continuous coefficients.}

\author[H. Dong]{Hongjie Dong}
\address[H. Dong]{Division of Applied Mathematics, Brown University,
182 George Street, Providence, RI 02912, United States of America}
\email{Hongjie\_Dong@brown.edu}
\thanks{H. Dong was partially supported by the NSF under agreement DMS-2055244.}

\author[D. Kim]{Dong-ha Kim}
\address[D. Kim]{Department of Mathematics, Yonsei University, 50 Yonsei-ro, Seodaemun-gu, Seoul 03722, Republic of Korea}
\email{skyblue898@yonsei.ac.kr}

\author[S. Kim]{Seick Kim}
\address[S. Kim]{Department of Mathematics, Yonsei University, 50 Yonsei-ro, Seodaemun-gu, Seoul 03722, Republic of Korea}
\email{kimseick@yonsei.ac.kr}
\thanks{S. Kim is supported by the National Research Foundation of Korea (NRF) under agreement NRF-2022R1A2C1003322.}



\begin{abstract}
This paper investigates the Dirichlet problem for a non-divergence form elliptic operator $L$ in a bounded domain of $\mathbb{R}^d$.
Under certain conditions on the coefficients of $L$, we first establish the existence of a unique Green's function in a ball and derive two-sided pointwise estimates for it.
Utilizing these results, we demonstrate the equivalence of regular points for $L$ and those for the Laplace operator, characterized via the Wiener test.
This equivalence facilitates the unique solvability of the Dirichlet problem with continuous boundary data in regular domains.
Furthermore, we construct the Green's function for $L$ in regular domains and establish pointwise bounds for it.
This advancement is significant, as it extends the scope of existing estimates to domains beyond $C^{1,1}$, contributing to our understanding of elliptic operators in non-divergence form.
\end{abstract}
\maketitle

\section{Introduction}

We consider an elliptic operator $L$ defined as follows:
\begin{equation}			\label{master_eq}
Lu=  a^{ij} D_{ij}u+ b^i D_i u+cu
\end{equation}
in $\bR^d$, where $d \ge 3$.
Throughout this paper, we adhere to the standard summation convention over repeated indices.

We assume that the principal coefficients matrix $\mathbf{A}=(a^{ij})$ is symmetric, satisfies the uniform ellipticity condition, and is of Dini mean oscillation, i.e., the mean oscillation of $\mathbf A$ satisfies the Dini criterion.
Moreover, we require $c \le 0$, $\vec b=(b^1,\ldots, b^d) \in L^{p_0}_{\rm loc}(\bR^d)$,  $c \in L^{p_0/2}_{\rm loc}(\bR^d)$, for some $p_0>d$.
Refer to Conditions \ref{cond1} and \ref{cond2} in Section \ref{sec2} for more details.

This paper focuses on the Dirichlet problem for the equation $Lu=0$ in $\Omega \subset \mathbb{R}^d$ with continuous boundary data $u=f$ on $\partial \Omega$.
Concerning second-order elliptic equations in divergence form, a result by Littman, Stampacchia, and Weinberger \cite{LSW63} indicates that regular points coincide with those of Laplace's equation (see also \cite{GW82}).
Similarly, for second-order elliptic equations in non-divergence form, this holds true if the coefficients are H\"older or Dini continuous (see \cite{Herve, Krylov67}).
However, without such assumptions, counterexamples exist, even if the coefficients are uniformly continuous in $\mathbb{R}^d$ (see \cite{Bauman84a, Miller, KY17}).

The characterization of regular boundary points is a significant problem in the theory of elliptic equations, drawing the attention of numerous researchers.
Wiener solved this problem for Laplace's equation, providing the answer in the form of the so-called Wiener's test, which involves the sum of certain capacities.
In \cite{Bauman85}, Wiener's result has been extended to elliptic operators in non-divergence form with continuous coefficients, where a Wiener-type criterion is provided for the regularity of points.
Nonetheless, no new insights were provided regarding the alignment of regular points for the operator and the Laplacian.

The results in \cite{LSW63, GW82} are quite notable, as they highlight a significant difference in regularity requirements between uniformly elliptic equations in divergence form and those in non-divergence form. While no regularity condition beyond measurability is necessary for equations in divergence form, equations in non-divergence form typically require some restrictions on the modulus of continuity. This distinction arises primarily from the fact that the adjoint operator for divergence form equations remains in divergence form, whereas the adjoint operator for non-divergence form equations, often denoted as double divergence form operator, is fundamentally different.

It has long been recognized that weak solutions to the double divergence form equation $D_{ij}(a^{ij}u)=0$ are not necessarily continuous, even if the coefficients $a^{ij}$ are uniformly continuous and  elliptic (see \cite{Bauman84b}).
However, if $a^{ij}$ are Dini continuous, then a weak solution of $D_{ij}(a^{ij}u)=0$ is continuous.
More precisely, a weak solution has a continuous representative. In particular, if $a^{ij}$ are H\"older continuous, then a solution is H\"older continuous (see \cite{Sjogren73, Sjogren75}).

The main tool for establishing the equivalence of regular points for the second-order uniformly elliptic equations in divergence form and the Laplace's operator in \cite{LSW63, GW82} is the pointwise estimates for the Green's function $G(x,y)$, where it is shown that  $G(x,y) \le N \abs{x-y}^{2-d}$.
Unlike the Green's function for uniformly elliptic operators in divergence form, the Green's function for non-divergent elliptic operators does not necessarily satisfy this pointwise bound (see \cite{Bauman84a}). Notably, the function $G^*(x,y):=G(y,x)$ serves as the Green's function for the adjoint operator.
Viewing it from this angle, it is not unexpected that the existence of unbounded solutions to the double divergence form equation $D_{ij}(a^{ij}u)=0$ with continuous coefficients $a^{ij}$ as presented in \cite{Bauman84b} and the demonstration of a counter-example for the pointwise bound for the Green's function of the non-divergence form operator $a^{ij}D_{ij}$ with continuous coefficients provided in \cite{Bauman84a} are closely intertwined.

In recent papers \cite{DK17, DEK18}, it has been shown that weak solutions to the double divergence form equation, $L^*u=D_{ij}(a^{ij}u)-D_i(b^i u)+cu=0$, are continuous, accompanied by estimates on their modulus of continuity, provided that the coefficients satisfy the conditions outlined at the beginning of the paper.
This outcome has facilitated the development of the Green's function for the non-divergence form elliptic operator (with $b^i= c= 0$) in $C^{1,1}$ domains, enabling the establishment of pointwise bounds comparable to those of the Green's function for Laplace's equation (see \cite{HK20, DK21}).

In this paper, we first construct the Green's function for $L$, as defined by \eqref{master_eq} and subject to Conditions \ref{cond1} and \ref{cond2}, in a ball, and establish two-sided pointwise estimates for the Green's function, as outlined in Theorem~\ref{thm_green_function}.
Given the crucial role of pointwise estimates for the Green's function in establishing the equivalence of regular points between Laplace's equation and uniformly elliptic second-order equations in divergence form, it is reasonable to expect that regular points for the non-divergence form operator $L$ align with those for the Laplace operator.
Indeed, we demonstrate that under the assumption that the coefficients of $L$ satisfy Conditions \ref{cond1} and \ref{cond2}, the regular points for $L$ coincide with those for the Laplace operator, as stated in Theorem \ref{cor0800sat}.
This achievement allows us to establish the unique solvability of the Dirichlet problem with continuous boundary data in regular domains, as depicted in Theorem \ref{thm0802sat}. Furthermore, we present the construction of the Green's function for the operator $L$ in regular domains and establish pointwise bounds for this Green's function, as detailed in Theorem \ref{thm1127sat}.
This progress is significant, as such an estimate was previously only available for the Green's function for operators without lower-order terms in $C^{1,1}$ domains.

The organization of the paper is as follows:
In Section \ref{sec_main}, we summarize the main results.
Section \ref{sec2} introduces the conditions on the coefficients of the non-divergence form operator $L$ and presents preliminary results leading to the construction of Perron's solution in an arbitrary domain.
Section \ref{sec3} focuses on constructing the Green's function for $L$ in a ball and establishing two-sided pointwise bounds, which serve as key tools for further developments.
In Section \ref{sec4}, we introduce the concepts of potential and capacity adapted to the non-divergence form operator $L$ and present a theorem demonstrating the comparability of the capacity for $L$ with that of the Laplace operator.
Section \ref{sec5} presents one of the main results, namely, that a boundary point is regular if and only if the Wiener test holds at that point.
Section \ref{sec6} covers the construction of the Green's function in regular domains by utilizing the solvability of the Dirichlet problem in these domains, established in the previous section.
Section \ref{sec7} serves as an appendix, where we provide proofs for technical lemmas.

Finally, an additional remark is in order:
We do not address the two-dimensional problem here, as we plan to address it in a separate paper.
The Green's function for two-dimensional domains requires different techniques, as they exhibit logarithmic singularity at the pole (see \cite{DK21}).

\section{Main Results}			\label{sec_main}
The following theorems summarize our main results. While they are repeated from the main text, we present them here for improved readability and convenience.

\begin{main}[Theorem \ref{cor0800sat}]
Assume that Conditions \ref{cond1} and \ref{cond2} hold.
Let $\Omega \subset \bR^d$ be a bounded open set.
A point $x_0 \in \partial \Omega$ is a regular point for $L$ (See Definition \ref{def_regpt}) if and only if $x_0$ is a regular point for the Laplace operator.
\end{main}

\begin{main}[Theorem \ref{thm0802sat}]
Under the assumptions that Conditions \ref{cond1} and \ref{cond2} are satisfied, consider a bounded regular domain $\Omega \subset \bR^d$ (see Definition \ref{def_reg}).
For $f \in C(\partial\Omega)$, the Dirichlet problem,
\[
Lu=0 \;\text{ in }\;\Omega, \quad u=f \;\text{ on }\;\partial\Omega,
\]
possesses a unique solution in $W^{2,p_0/2}_{\rm loc}(\Omega)\cap C(\overline\Omega)$.
\end{main}

\begin{main}[Theorem	\ref{thm1127sat}]
Assuming Conditions \ref{cond1} and \ref{cond2}, let $\Omega \subset \bR^d$ be a bounded regular domain.
Then there exists a Green's function $G_{\Omega}(x,y)$ in $\Omega$, and it has the following pointwise bound:
\[
0\le G_{\Omega}(x,y) \le N \abs{x-y}^{2-d},\quad \forall x\neq y \in \Omega,
\]
where $N$ is a constant depending only on $d$, $\lambda$, $\Lambda$, $\omega_{\mathbf A}$, $p_0$, and $\diam \Omega$.
\end{main}

\section{Preliminary}			\label{sec2}
We choose an open ball $\mathcal{B} \subset \mathbb{R}^d$ containing $\overline \Omega$, where $\Omega$ is a domain under consideration.
To aid our analysis, we introduce additional open balls denoted as $\mathcal{B}'$ and $\mathcal{B}''$, both concentric with $\mathcal{B}$ and satisfying $\Omega \Subset \mathcal{B}'' \Subset \mathcal{B}' \Subset \mathcal{B}$.
Throughout the paper, these open balls $\mathcal{B}$, $\mathcal{B}'$, and $\mathcal{B}''$ remain fixed. As an example, one could choose $\mathcal{B}=B_{4R}(0)$, $\mathcal{B}'=B_{2R}(0)$, and $\mathcal{B}''=B_R(0)$, where $R$ is selected to be sufficiently large.

\subsection{Assumptions on the coefficients}	\label{sec2.1}

\begin{condition}			\label{cond1}
The coefficients of $L$ are measurable and defined in the whole space $\bR^d$.
The principal coefficients matrix $\mathbf{A}=(a^{ij})$ is symmetric and satisfies the ellipticity condition:
\[
\lambda \abs{\xi}^2 \le \mathbf{A}(x) \xi  \cdot \xi \le \lambda^{-1} \abs{\xi}^2,\quad \forall x \in \bR^d,\;\;\forall \xi \in \bR^d,
\]
where $\lambda \in (0,1]$ is a constant.
The lower-order coefficients $\vec b=(b^1,\ldots, b^d)$ and $c$ belong to $L^{p_0/2}_{\rm loc}(\bR^d)$ and $L^{p_0}_{\rm loc}(\bR^d)$ for some $p_0>d$, and
\[
\norm{\vec b}_{L^{p_0}(\mathcal{B})}+ \norm{c}_{L^{p_0/2}(\mathcal{B})} \le \Lambda,
\]
where $\Lambda=\Lambda(\mathcal{B})<\infty$.
Additionally, we require that $c\le 0$.
\end{condition}

\begin{condition}			\label{cond2}
The mean oscillation function $\omega_{\mathbf A}: \bR_+ \to \bR$ defined by
\[
\omega_{\mathbf A}(r):=\sup_{x\in \mathcal{B}} \fint_{\mathcal{B} \cap B_r(x)} \,\abs{\mathbf A(y)-\bar {\mathbf A}_{x,r}}\,dy, \;\; \text{where }\bar{\mathbf A}_{x,r} :=\fint_{\mathcal{B} \cap B_r(x)} \mathbf A,
\]
satisfies the Dini condition, i.e.,
\[
\int_0^\epsilon \frac{\omega_{\mathbf A}(t)}t \,dt <+\infty
\]
for some $\epsilon>0$.
\end{condition}

We denote by $C$ generic constants that depend solely on $d$, $\lambda$, $\Lambda$, $p_0$, $\omega_{\mathbf A}$, and the diameter of $\mathcal{B}$, unless stated otherwise. The notation $A \lesssim B$ is used to indicate the existence of a generic constant $C > 0$ such that $A \leq CB$. Similarly, $A \simeq B$ is used when both $A \lesssim B$ and $B \lesssim A$ hold.

At times, we will make the additional assumption that $c\equiv 0$.
However, we now address the dispensability of this seemingly strong condition within the framework of the Dirichlet problem.
The following proposition demonstrates the existence of a strictly positive bounded solution in $W^{2,p_0/2}(\mathcal{B})$.

\begin{proposition}			\label{lem1012thu}
Under Conditions \ref{cond1} and \ref{cond2}, there exists a function $\zeta \in W^{2,p_0/2}(\mathcal{B})$ satisfying
\begin{equation}			\label{eq1532tue}
L\zeta=0 \;\text{ in }\; \mathcal{B},\quad \zeta =1\; \text{ on }\;\partial \mathcal{B}.
\end{equation}
Furthermore, $\zeta$ is continuous and satisfies the estimates $\delta \le \zeta \le 1$ in $\mathcal{B}$, along with  $\norm{D\zeta}_{L^{p_0}(\mathcal{B})} \le  C$, where $\delta$ and $C$ are positive constants dependent only on $d$, $\lambda$, $\Lambda$, $p_0$, $\omega_{\mathbf A}$, and the diameter of $\mathcal{B}$.
\end{proposition}

\begin{proof}
By \cite[Theorem 4.2]{Krylov21}, there exists a unique function $v \in W^{2,p_0/2}(\mathcal{B})\cap W^{1,p_0/2}_0(\mathcal{B})$ that satisfies $Lv=-c$ in $\mathcal{B}$.
Then, $\zeta=1+ v \in W^{2,p_0/2}(\mathcal{B})$, and it satisfies \eqref{eq1532tue}.
It is evident from the Sobolev embedding that $\zeta \in C(\overline{\mathcal{B}})$.
We will demonstrate that
\begin{equation}			\label{eq2100thu}
0 \le \zeta \le 1 \quad\text{in}\quad \mathcal{B}.
\end{equation}
Let $\set{\vec b_n}$ and $\set{c_n}$ be sequences smooth functions converging to $\vec b$ in $L^{p_0}(\mathcal{B})$ and to $c$ in $L^{p_0/2}(\mathcal{B})$, respectively. Additionally, we assume $c_n \le 0$ in $\mathcal{B}$.

Consider the operators $L^{(n)}=a^{ij}D_{ij}+b_n^i D_i + c_n$.
By \cite[Theorem 4.2]{Krylov21}, there exists a unique $v_n \in W^{2,p}(\mathcal{B})\cap W^{1,p}_0(\mathcal{B})$ that satisfies $L^{(n)}v_n=-c_n$ in $\mathcal{B}$.
Then, it follows from \cite[Theorem 1.5]{DEK18} that $v_n \in C^2(\overline{\mathcal{B}})$.
Consequently, $\zeta_n=1+v_n$ becomes the classical solution to $L^{(n)}\zeta_n=0$ in $\mathcal{B}$, with $\zeta_n =1$ on $\partial \mathcal{B}$.
Applying the classical maximum principle, we establish $0\le \zeta_n \le 1$.
As
\[
a^{ij}D_{ij}(\zeta-\zeta_n)+b^i_n D_i(\zeta-\zeta_n)+c_n(\zeta-\zeta_n)=(b^i_n-b^i)D_i\zeta + (c_n-c)\zeta,
\]
and the right-hand side converges in $L^{p_0/2}(\mathcal{B})$ as $n\to \infty$, we deduce that $\zeta_n \to \zeta$ in $W^{2,p_0/2}(\mathcal{B})$.
Therefore, by the Sobolev embedding theorem, we conclude that $\zeta_n$ converges to  $\zeta$ uniformly in $\mathcal{B}$.
This confirms the claim \eqref{eq2100thu}.

Moving forward, given that $\zeta=1$ on $\partial \mathcal{B}$ and $\zeta \in C^{1-d/p_0}(\mathcal{B})$ by the Sobolev embedding, there is $\delta_0>0$ such that $\zeta(x) \ge \frac12$ whenever $\dist(x, \partial \mathcal{B})<\delta_0$.
The subsequent lemma establishes Harnack inequality in a small ball.

\begin{lemma}			\label{lem2151thu}
Assuming Conditions \ref{cond1} and \ref{cond2} hold, there exist positive constants $r_0$ and $N$, depending only on $d$, $\lambda$, $\Lambda$, $p_0$, and $\omega_{\mathbf A}$, such that if $u \in W^{2,p_0/2}(B_r)$ is a nonnegative solution of $Lu=0$ in $B_r \subset \mathcal{D}$ with $r<r_0$, then we have
\[
\sup_{B_{r/2}} u \le N \inf_{B_{r/2}}u.
\]
\end{lemma}
\begin{proof}
First we show that there exists $r_0>0$ such that if $v \in W^{2,p_0/2}(B_r)$ is a solution to the problem
\begin{equation}			\label{eq1847thu}
\left\{
\begin{array}{ccc}
Lv=0 &\text{ in } & B_r,\\
v =1 & \text{ on }&\partial B_r,
\end{array}
\right.
\end{equation}
where $B_r \subset \mathcal{B}$ and $0<r<r_0$, then $v \ge 1/2$ in $B_r$.

Choose $q$ such that $d/2<q<p_0/2$.
By \cite[Theorem 4.2]{Krylov21}, we obtain $\psi \in W^{2,q}(B_1)\cap W^{1,q}_0(B_1)$ as a solution to
\[
a^{ij}(r\,\cdot\,)D_{ij}\psi+r b^i(r\,\cdot\,)D_i \psi+r^{2}c(r\,\cdot\,)\psi=-r^{2}c(r\,\cdot\,)\; \text{ in }\; B_1,
\]
where
\[
\norm{rb^i(r\,\cdot\,)}_{L^p(B_1)}\le r^{1-d/p}\norm{b^i}_{L^p(\mathcal{B})}\quad\text{and}\quad
\norm{r^{2}c(r\,\cdot\,)}_{L^{p/2}(B_1)}\le r^{2(1-d/p)} \norm{c}_{L^{p/2}(\mathcal{B})}.
\]
Thus we have the estimate
\begin{equation}		\label{eq1725thu}
\norm{\psi}_{W^{2,q}(B_1)} \le C \norm{r^2c(r\,\cdot\,)}_{L^{q}(B_1)}.
\end{equation}
By the Sobolev embedding and H\"older's inequality, we derive from \eqref{eq1725thu} that
\begin{align}			\nonumber
\sup_{B_1} \,\abs{\psi} \le C \norm{\psi}_{W^{2,q}(B_1)} & \le  C r^2 \norm{c(r\,\cdot\,)}_{L^{q}(B_1)} \le C r^2  \norm{c(r\,\cdot\,)}_{L^{p_0/2}(B_1)}\\
					\label{eq2104sun}
& \le C r^{2-2d/p_0}\norm{c}_{L^{p_0/2}(\mathcal{B})} \le Cr^{2-2d/p_0}.
\end{align}
Since $p_0>d$, we can find $r_0>0$ such that the term $Cr^{2-2d/p_0}$ in \eqref{eq2104sun} is less than $1/2$ for $r<r_0$.
As $1+\psi(r^{-1}\,\cdot\,) \in W^{2,q}(B_r)$ is a solution to the problem \eqref{eq1847thu}, we have $v=1+\psi(r^{-1}\,\cdot\,)$ by the uniqueness.
Consequently, $v \ge 1/2$ in $B_r$ as claimed.

Next, we set $w=u/v$ and note that
\[
L u=v (a^{ij}D_{ij}w +  \tilde b^i D_i w)+wLv,\quad\text{where }\;\tilde b^i:=b^i+2a^{ij}D_j v/v.
\]
Given that $Lu=0=Lv$ in $B_r$, if we set $\tilde L=a^{ij}D_{ij}+\tilde b^i D_i$, then $\tilde L w=0$ in $B_r$.
Note that
\[
\norm{Dv}_{L^d(B_r)} = r^{-1}\norm{D\psi(r^{-1}\,\cdot\,)}_{L^d(B_r)} = \norm{D\psi}_{L^d(B_1)} \le C \norm{\psi}_{W^{2,q}(B_1)}  \le  C r^{2-2d/p_0},
\]
where we used the Sobolev inequality and  \eqref{eq2104sun}.
Therefore, we have
\[
\norm{\tilde{\vec b}}_{L^d(B_r)} \le \norm{\vec b}_{L^d(B_r)}+4 \lambda^{-1} \norm{Dv}_{L^d(B_r)} \le C.
\]
Consequently, by \cite[Theorem 3.1]{Safonov}, the Harnack inequality holds for $w$, i.e.,
\[
\sup_{B_{r/2}} w \le N \inf_{B_{r/2}} w.
\]
Since $u=vw$ and $1/2 \le v \le 1$ in $B_r$, the preceding inequality also holds for $u$.
\end{proof}

Applying Lemma~\ref{lem2151thu} to $\zeta$ in a chain of balls, we conclude that $\zeta \ge \delta$ in $\mathcal{B}$ for some positive constant $\delta>0$.
This completes the proof of the proposition.
\end{proof}

Notice that $u=\zeta v$, where $\zeta$ is as in Proposition \ref{lem1012thu}, satisfies
\begin{equation}
\label{dongha1654}
Lu=\zeta \left(a^{ij}D_{ij}v +  (b^i+2a^{ij}D_j \zeta/\zeta) D_i v\right)=:\zeta (a^{ij}D_{ij}v + \tilde b^i D_i v)=:\zeta \tilde{L}v.
\end{equation}
Hence, solving the Dirichlet problem in a domain $\Omega \Subset \mathcal{B}$ for $Lu = 0$ with boundary condition $u = f$ is equivalent to solving the Dirichlet problem
\begin{equation}\label{dongha1655}
    \left\{
\begin{array}{ccc}
\tilde{L}v=0 &\text{ in } & \Omega,\\
v=f/\zeta& \text{ on }&\partial \Omega.
\end{array}
\right.
\end{equation}
Moreover, Proposition \ref{lem1012thu} implies that $\tilde{\vec{b}} = (\tilde{b}^1,\ldots, \tilde{b}^d)$ satisfies $\norm{\tilde{\vec b}}_{L^{p_0}(\mathcal{B})} \le C$.
Consequently, in the sequel we may assume that $c\equiv 0$ without compromising any essential aspects.
Further details are provided in Theorem \ref{thm0802sat}.

\subsection{Dirichlet problem in a ball}
We begin by establishing the solvability of the Dirichlet problem with continuous boundary data in sufficiently smooth domains, such as balls.
\begin{lemma}			\label{lem0114mon}
Assume that Conditions \ref{cond1} and \ref{cond2} hold, and further suppose that $c = 0$.
Let $B \Subset \mathcal{B}$.
Then, for any $f \in C(\partial B)$, there exists a unique solution $u \in W^{2,p_0}_{\rm loc}(B) \cap C(\overline B)$ of the Dirichlet problem
\begin{equation}			\label{eq0853thu}
\left\{
\begin{array}{ccc}
Lu=0 &\text{ in } & B,\\
u=f& \text{ on }&\partial B.
\end{array}
\right.
\end{equation}
Additionally, we have
\begin{equation}			\label{eq1025thu}
\max_{\overline B}\, \abs{u} \le \max_{\partial B} \,\abs{f}.
\end{equation}
\end{lemma}
\begin{proof}
To establish the uniqueness assertion, it is enough to show that if $u \in W^{2,p_0}_{\rm loc}(B) \cap C(\overline B)$ satisfies of $Lu=0$ in $B$ and $u=0$ on $\partial B$, then $u=0$.
Indeed, this is an immediate consequence of \cite[Theorem 1.1]{Safonov}.

To get the existence, let $\set{p_n}$ be a sequence of polynomials converging to $f$ uniformly on $\partial B$.
Since $Lp_n \in L^{p_0}(B)$, by \cite[Theorem 4.2]{Krylov21}, there exists a unique $v_n \in W^{2,p_0}(B)\cap W^{1,p_0}_0(B)$ that satisfies $L v_n=-Lp_n$ in $B$.
Then $u_n:=v_n+p_n$ satisfies $L u_n=0$ in $B$ and $u_n=p_n$ on $\partial B$.
The differences, $u_n-u_m$, clearly satisfies
\[
L(u_n-u_m)=0\;\text{ in }\;B,\quad u_n-u_m= p_n-p_m\;\text{ on }\;\partial B.
\]
By \cite[Theorem 1.1]{Safonov}, it follows that
\begin{equation}			\label{eq1026thu}
\max_{\overline B}\, \abs{u_n} \le \max_{\partial B} \,\abs{p_n},
\end{equation}
and also that $\set{u_n}$ is a Cauchy sequence in $C(\overline B)$.
Then, by \cite[Theorem 4.2]{Krylov21} combined with standard iteration method, for any $B' \Subset B$, we deduce
\begin{equation}			\label{eq1057thu}
\norm{u_n-u_m}_{W^{2,p_0}(B')} \le C\norm{u_n-u_m}_{L^{p_0}(B)},
\end{equation}
where the constant $C$ depends on $d$, $\nu$, $p_0$, $\omega_{\mathbf A}$, $B'$, and $B$.
Therefore, we deduce that $\set{u_n}$ converges in $W^{2,p_0}_{\rm loc}(B) \cap C(\overline B)$ to a solution of the Dirichlet problem \eqref{eq0853thu}.
The inequality \eqref{eq1025thu} follows from \eqref{eq1026thu} by taking the limit $n\to \infty$.
\end{proof}

Assume Conditions \ref{cond1} and \ref{cond2}, and further suppose $c = 0$.
For any $f \in C_c(\mathcal{B})$, and $B \Subset \mathcal{B}$, let $u \in W^{2, p_0}_{\rm loc}(B)\cap C(\overline B)$ be the solution of the Dirichlet problem
\[
\left\{
\begin{array}{ccc}
Lu=0 &\text{ in } & B,\\
u=f& \text{ on }&\partial B.
\end{array}
\right.
\]
Then for $x\in B$, the mapping $f \mapsto u(x)$ is a linear functional on $C_c(\mathcal{B})$ since
\[
\abs{u(x)} \le \max_{\overline B}\, \abs{u} \le \max_{\partial B}\, \abs{f} \le \max_{\mathcal{B}}\, \abs{f}.
\]
Moreover, it is a positive since $u(x) \ge 0$ if $f\ge 0$ by \cite[Theorem 1.1]{Safonov}.
Therefore, there exists a Radon measure $\omega^x_{B}$ on $\mathcal{B}$, which corresponds to the mapping.
It is clear that support of $\omega^x_{B}$ lies in $\partial B$ so that we have
\[
u(x)=\int_{\partial B} f\,d\omega^x_{B}.
\]
We call this measure $\omega^x_{B}$ the $L$-harmonic measure at $x$ in $B$.

\subsection{Supersolutions,  subsolutions, and Perron's method}
For the rest of this section, assume Conditions \ref{cond1} and \ref{cond2} hold, with $c = 0$.
\begin{definition}		\label{def_supersol}
Let $\mathcal{D} \subset \mathcal{B}$ be an open set.
An extended real-valued function $u$ is said to be an $L$-supersolution in $\mathcal{D}$ if the following conditions are satisfied:
\begin{enumerate}
\item[(i)]
$u$ is not identically $+\infty$ in any connected component of $\mathcal{D}$.
\item[(ii)]
$u>-\infty$ in $\mathcal{D}$.
\item[(iii)]
$u$ is lower semicontinuous in $\mathcal{D}$.
\item[(iv)]
For any ball $B \Subset \mathcal{D}$ and $x \in B$,  we have
\[
u(x) \ge \int_{\partial B} u\,d\omega^x_{B}.
\]
\end{enumerate}
An $L$-subsolution in $\mathcal{D}$ is defined in similarly, such that $u$ is an $L$-subsolution if and only if $-u$ is an $L$-supersolution.
If  $u$ is both $L$-supersolution and $L$-subsolution in $\mathcal{D}$, then we say that it is an $L$-solution in $\mathcal{D}$.
\end{definition}

\begin{lemma}				\label{lem_solution}
A function $u$ is an $L$-solution in $\mathcal{D}$ if and only if $u \in W^{2,p_0}_{\rm loc}(\mathcal{D})$ and it satisfies  $Lu=0$ a.e. in $\mathcal{D}$.
\end{lemma}
\begin{proof}
Suppose $u$ is an $L$-solution in $\mathcal{D}$.
Then $u$ is continuous in $\mathcal{D}$, and for any ball $B \Subset \mathcal{D}$, we have
\begin{equation}			\label{eq0925thu}
u(x)=\int_{\partial B} u \,d\omega_{B}^x,\quad \forall x \in B.
\end{equation}
Therefore, by Lemma \ref{lem0114mon}, we have $u \in W^{2,p_0}_{\rm loc}(B)\cap C(\overline B)$ and it satisfies $Lu=0$ in $B$.
Since $B\Subset \mathcal{D}$ is arbitrary, we conclude that $u \in W^{2,p_0}_{\rm loc}(\mathcal{D})$ and $Lu=0$ a.e. in $\mathcal{D}$.
The converse is also true as $u \in C(\mathcal{D})$, and \eqref{eq0925thu} holds for any $B \Subset \mathcal{D}$.
\end{proof}

\begin{definition}
We  denote
\begin{align*}
\mathfrak{S}^+(\mathcal{D})&=\set{u: u \text{ is an $L$-supersolution in $\mathcal{D}$}},\\
\mathfrak{S}^-(\mathcal{D})&=\set{u: u \text{ is an $L$-subsolution in $\mathcal{D}$}}.
\end{align*}
\end{definition}

\begin{lemma}			\label{lem2334sun}
Let $u \in \mathfrak{S}^+(\mathcal{D})$. The following statements are valid:
\begin{enumerate}
\item (Strong minimum principle)
Let $\mathcal{D}$ be connected.
If $u \in \mathfrak{S}^+(\mathcal{D})$, then the infimum of $u$ is not attained in $\mathcal{D}$ unless $u$ is constant on $\mathcal{D}$.
\item (Comparison principle)
If $u \in \mathfrak{S}^+(\mathcal{D})$ and $\liminf_{y \to x,\; y\in \mathcal{D}} u(y) \ge 0$ for every $x \in \partial \mathcal{D}$, then $u \ge 0$ in $\mathcal{D}$.
\item
If $u \in \mathfrak{S}^+(\mathcal{D})$, then for every $x_0 \in \mathcal{D}$, we have $\liminf_{x \to x_0} u(x)=u(x_0)$.
\item
If $u$, $v \in \mathfrak{S}^+(\mathcal{D})$, and $c>0$, then $cu$, $u+v$, and $\min(u,v) \in \mathfrak{S}^+(\mathcal{D})$.
\item
(Pasting lemma)
For $v \in \mathfrak{S}^+(\mathcal{D}')$, where $\mathcal{D}'$ is an open subset of $\mathcal D$, define $w$ by
\[
w= \begin{cases}
\min(u,v)& \text{in }\;\mathcal{D}',\\
\phantom{\min}u &  \text{in }\; \mathcal{D} \setminus \mathcal{D}'.
\end{cases}
\]
If $w$ is lower semicontinuous in $\partial \mathcal{D}'$, then $w \in \mathfrak{S}^+(\mathcal{D})$.
\end{enumerate}
\end{lemma}
\begin{proof}
Refer to the Appendix for the proof.
\end{proof}

\begin{definition}
For $u \in \mathfrak{S}^+(\mathcal{D})$ and a ball $B\Subset \mathcal{D}$, we define $\mathscr{E}_{B}u$ as follows:
\[
\mathscr{E}_{B}u(x) = \begin{cases}
\phantom{\int_{\partial B}}u(x)& \text{in }\;\mathcal{D} \setminus  B,\\
\int_{\partial B} u\,d\omega^x_{B} &  \text{in }\; B.
\end{cases}
\]
The function $\mathscr{E}_{B}u$ is called the lowering of $u$ over $B$.
\end{definition}

\begin{lemma}			\label{lem1033thu}
The following properties hold:
\begin{enumerate}
\item
$\mathscr{E}_{B}u \le u$.
\item
$\mathscr{E}_{B}u$ is an $L$-supersolution in $\mathcal{D}$.
\item
$\mathscr{E}_{B}u$ is an $L$-solution in $B$.
\end{enumerate}
\end{lemma}
\begin{proof}
Refer to the Appendix for the proof.
\end{proof}

\begin{definition}
A collection $\mathscr F$ of $L$-supersolutions is said to be saturated in $\mathcal{D}$ if $\mathscr F$ satisfies the following conditions:
\begin{enumerate}
\item[(i)]
If $u$, $v \in \mathscr F$, then $\min(u,v) \in \mathscr F$.
\item[(ii)]
$\mathscr{E}_{B}u \in \mathscr F$ whenever $u \in \mathscr F$ and $B \Subset \mathcal{D}$.
\end{enumerate}
\end{definition}

\begin{lemma}			\label{lem1034thu}
Let $\mathscr{F}$ be a saturated collection of $L$-supersolutions in $\mathcal{D}$, and let
\[
v(x) := \inf_{u \in \mathscr{F}} u(x).
\]
If $v$ does not take the value $-\infty$ in $\mathcal{D}$, then $Lv=0$ in $\mathcal{D}$.
\end{lemma}
\begin{proof}
Refer to the Appendix for the proof.
\end{proof}

For a continuous function $f$ on $\partial\mathcal{D}$, consider the Dirichlet problem
\[
\left\{
\begin{array}{ccc}
Lu=0 &\text{ in } & \mathcal{D},\\
u=f& \text{ on }&\partial \mathcal{D}.
\end{array}
\right.
\]
\begin{definition}			\label{def_perron}
We define an upper Perron's solution $\overline H_f$ and a lower Perron's solution $\underline H_f$ of the Dirichlet problem by
\begin{align*}
\overline H_f(x)=\inf\, \Set{v(x): v \in \mathfrak{S}^+(\mathcal{D}), \;\liminf_{y\to x_0,\, y \in \mathcal{D}}\, v(y) \ge f(x_0),\;\forall x_0 \in \partial\mathcal{D}},\\
\underline H_f(x)=\sup\, \Set{v(x): v \in \mathfrak{S}^-(\mathcal{D}), \;\limsup_{y\to x_0, \,y \in \mathcal{D}}\, v(y) \le f(x_0),\;\forall x_0\in \partial\mathcal{D}}.
\end{align*}
\end{definition}
Note that
\begin{equation}			\label{eq1208thu}
\underline H_{f}=-\overline H_{-f}.
\end{equation}

\begin{lemma}			\label{lem1210thu}
For a continuous function $f$ on $\partial\mathcal{D}$, the upper and lower Perron's solutions $\overline H_f$ and  $\underline H_f$ are both $L$-solutions.
\end{lemma}
\begin{proof}
First, we observe that the constant function $M := \max_{\partial \mathcal{D}}\,\abs{f}$ belongs to $\mathscr{F}$, defined as
\[
\mathscr F:=\Set{v \in \mathfrak{S}^+(\mathcal{D}): \liminf_{y\to x_0,\, y \in \mathcal{D}}\, v(y) \ge f(x_0),\;\forall x_0\; \text{ on }\;\partial\mathcal{D}}.
\]
Moreover, by the comparison principle, it follows that $v \geq -M$ for every $v \in \mathscr{F}$.
Hence, we have $-M \leq \overline{H}_f \leq M$.
Leveraging Lemmas \ref{lem_solution}, \ref{lem1033thu}, and \ref{lem1034thu}, we conclude that $\overline H_f$ is an $L$-solution because the collection $\mathscr{F}$ is clearly saturated.
Additionally, from \eqref{eq1208thu}, we deduce that $\underline H_f$ is also an $L$-solution.
\end{proof}
\begin{definition}
                \label{def_regpt}
A point $x_0$ on $\partial \mathcal{D}$ is called a regular point if, for all $f \in C(\partial \mathcal{D})$, we have
\begin{equation}			\label{eq1212thu}
\lim_{x \to x_0, \, x\in \mathcal{D}}\, \underline H_f(x)=f(x_0).
\end{equation}
\end{definition}

In \eqref{eq1212thu} of the previous definition, it is possible to replace $\underline H_f$ with $\overline H_f$ due to identity  \eqref{eq1208thu}.

\section{Green's function in a ball}			\label{sec3}
Recall that $\mathcal{B} \subset \bR^d$ is a fixed open ball containing $\overline \Omega$.
In this section, we establish the existence of a Green's function in $\mathcal{B}$ and present its several properties.

\begin{theorem}\label{main_Green_const}
We assume the Conditions \ref{cond1} and \ref{cond2}.
Then, there exists a Green's function $G(x,y)$ of $L$ in $\mathcal{B}$ and the Green's function is unique in the following sense: if $v$ is the unique adjoint solution of the problem
\begin{equation}				\label{eq1738sun}
L^*v=D_{ij}(a^{ij}v)-D_i(b^iv)+cv=f\;\text{ in }\;\mathcal{B},\quad v=0\;\text{ on }\;\partial \mathcal{B},
\end{equation}
where $f \in L^p(\mathcal{B})$ with $p>d/2$, then $v$ is represented by
\begin{equation*}
v(y)=\int_\mathcal{B} G(x,y) f(x)\,dx.
\end{equation*}
Moreover, the following pointwise estimate holds:
\begin{equation*}		
G(x,y) \le C \abs{x-y}^{2-d},\quad  x \neq y \in \mathcal{B}.
\end{equation*}
Finally, 
$G^\ast(x,y):=G(y,x)$ becomes the Green's function for the adjoint operator $L^\ast$, which is characterized as follows:
for $q>d/2$ and $f\in L^{q}(\mathcal{B})$,
if $v \in W^{2,q}(\mathcal{B})\cap W^{1,q}_0(\mathcal{B})$ is the strong solution of
\begin{equation*}
-Lv=f \;\text{ in }\;\mathcal{B},\quad v=0\;\text{ on }\;\partial \mathcal{B},
\end{equation*}
then, we have the representation formula
\begin{equation*}
v(y)=\int_\mathcal{B} G^\ast(x,y)f(x)\,dx=\int_\mathcal{B} G(y,x)f(x)\,dx.
\end{equation*}
\end{theorem}
\begin{proof}
In \cite{HK20}, the Green's functions are constructed for the operator $L_0:= a^{ij}D_{ij}$ without lower-order coefficients in $C^{1,1}$ domains, a class that certainly includes $\mathcal{B}$.
To ensure that the Green's function is nonnegative, we invert the sign of the function constructed in \cite{HK20}. Thus, formally, we have:
\[
L_0 G_0(\,\cdot,y)=-\delta_y\;\text{ in }\;\mathcal{B}\quad\text{and}\quad G_0(\,\cdot,y)=0 \;\text{ on }\;\partial \mathcal{B}.
\]
To establish the Green's function for $L$ in $\mathcal{B}$, we address the following problem for each $y \in \mathcal{B}$:
\begin{equation}			\label{eq1123thu}
Lu= -b^iD_i G_0(\,\cdot,y) -c G_0(\,\cdot,y)\;\text{ in }\;\mathcal{B},\quad u=0 \;\text{ on }\;\partial\mathcal{B}.
\end{equation}
According to the pointwise estimates for the Green's function provided in \cite{HK20}, it is clear that for each $y \in \mathcal{B}$, we have
\[
b^i D_i G_0(\,\cdot, y),\; cG_0(\,\cdot, y) \in L^{p_1}(\mathcal{B}),
\]
for some $p_1>1$.
According to \cite[Theorem 4.2]{Krylov21}, there exists a unique solution $u=u^y \in W^{2,p_1}(\mathcal{B})\cap W^{1,p_1}_0(\mathcal{B})$ to the problem \eqref{eq1123thu}.
It is noteworthy that by H\"older's inequality and the Sobolev embedding theorem, we obtain
\begin{align}				\nonumber
\norm{u^y}_{L^{1}(B_r(x_0) \cap \mathcal{B})}  &\le C r^{2+d(1-1/p_1)} \norm{u^y}_{L^{p_1d/(d-2p_1)}(B_{r}(x_0)\cap \mathcal{B})} \le C r^{2+d(1-1/p_1)} \norm{u^y}_{L^{p_1d/(d-2p_1)}(\mathcal{B})}\\
						\label{eq1410thu}
&\le C r^{2+d(1-1/p_1)} \norm{u^y}_{W^{2,p_1}(\mathcal{B})} \le Cr^{2+d(1-1/p_1)},
\end{align}

Now, we will demonstrate that
\[
G(x,y):=G_0(x,y)+u^y(x)
\]
serves as the Green's function for $L$ in $\mathcal{B}$.
For any $f \in C^\infty_c(\mathcal{B})$, let $v \in L^{p_1'}(\mathcal{B})$ be the solution of \eqref{eq1738sun}.
By \cite[Theorem 1.8]{DEK18}, we find that $v \in C(\overline{\mathcal{B}})$.
Moreover, from the definition of the solution to the problem \eqref{eq1738sun}, we have
\begin{equation}			\label{eq1428wed}
\int_{\mathcal{B}} f w = \int_{\mathcal{B}} v Lw,\quad \forall w \in  W^{2,p_1}(\mathcal{B})\cap W^{1,p_1}_0(\mathcal{B}).
\end{equation}
Since $u^y \in W^{2,p_1}(\mathcal{B})\cap W^{1,p_1}_0(\mathcal{B})$ is a solution of \eqref{eq1123thu}, it follows that
\begin{equation}			\label{eq1805sun}
\int_{\mathcal{B}} f u^y = \int_{\mathcal{B}} v Lu^y = -\int_{\mathcal{B}} b^i D_i G_0(\,\cdot, y)v - \int_{\mathcal{B}} c G_0(\,\cdot, y)v.
\end{equation}
On the other hand, taking  $w=G_0^\epsilon(\,\cdot, y)$ in \eqref{eq1428wed}, where $G_0^\epsilon(\,\cdot, y)$ is the approximate Green's function of $L_0$ as considered in \cite{HK20} (with the sign inverted), we obtain 
\begin{equation}			\label{eq2103sat}
\int_{\mathcal{B}} f G_0^\epsilon(\,\cdot,y)=\int_{\mathcal{B}} v L G_0^\epsilon(\,\cdot,y)=
-\fint_{B_\epsilon(y)\cap \mathcal{B}} v + \int_{\mathcal{B}} b^i D_i G_0^\epsilon(\,\cdot, y)v+\int_{\mathcal{B}} c G_0^\epsilon(\,\cdot, y)v.
\end{equation}

\begin{lemma}			\label{lem1631wed}
There exists a sequence $\set{\epsilon_k}$ converging to zero such that
\begin{align*}
G^{\epsilon_k}_0(\,\cdot, y) &\rightharpoonup G_0(\,\cdot, y) \;\text { weakly in }\;L^{p}(\mathcal{B})\; \text{ for }\;1<p<d/(d-2),\\
D G^{\epsilon_k}_0(\,\cdot, y) &\rightharpoonup DG_0(\,\cdot, y) \;\text { weakly in } L^{p}(\mathcal{B})\; \text{ for }\;1<p<d/(d-1).
\end{align*}
\end{lemma}
\begin{proof}
The first part has already been established in \cite{HK20}, relying on the following uniform estimate, which holds for $0 < \epsilon < \diam \mathcal{B}$:
\[
\Abs{\set{x \in \mathcal{B} : \abs{G_0(x,y)} > t} } \le C t^{-\frac{d}{d-2}},\quad \forall t>0.
\]
Additionally, as noted in \cite{HK20}, we have the expression:
\[
G_0^\epsilon(x,y)=\fint_{B_\epsilon(y)\cap \mathcal{B}} G_0(x,z)\,dz.
\]
From this identity, combined with the continuity of $G_0(x,\cdot\,)$ in $\mathcal{B} \setminus \set{y}$ and the estimates (1.12) of \cite{HK20}, we deduce, for $x \neq y \in \mathcal{B}$, the following estimate, which holds uniformly for $0 < \epsilon < \frac13 \abs{x-y}$:
\[
\abs{D_x G_0^\epsilon(x,y)} \le C \abs{x-y}^{1-d}.
\]
By replicating essentially the same proof as Lemma 2.13 in \cite{HK20}, we obtain the following estimate, valid for $0 < \epsilon < \diam \mathcal{B}$:
\[
\int_{\mathcal{B}\setminus B_r(y)} \abs{D_x G_0^\epsilon(x,y)}^2\,dx \le C r^{2-d},\quad \forall r>0.
\]
Then, employing the same proof technique as Lemma 2.19 in \cite{HK20}, we find that for $0 < \epsilon < \text{diam}(\mathcal{B})$:
\[
\Abs{\set{x \in \mathcal{B} : \abs{DG_0(x,y)} > t} } \le C t^{-\frac{d}{d-1}},\quad \forall t>0.
\]
The second assertion follows easily from the uniform estimate above.
\end{proof}

Therefore, by taking the limit $\epsilon_k \to 0$ in \eqref{eq2103sat}, we obtain
\begin{equation}			\label{eq1806sun}
\int_{\mathcal{B}} f G_0(\,\cdot,y) =-v(y) + \int_{\mathcal{B}} b^i D_i G_0(\,\cdot, y)v+\int_{\mathcal{B}} c G_0(\,\cdot, y)v,
\end{equation}
where we utilized Lemma \ref{lem1631wed} and the fact that  $v \in C(\overline{\mathcal{B}})$.
By combining \eqref{eq1805sun} and \eqref{eq1806sun}, we deduce
\[
\int_{\mathcal{B}} f G(\,\cdot, y)= \int_{\mathcal{B}} f G_0(\,\cdot, y) + \int_{\mathcal{B}} f u^y= -v(y).
\]
Therefore, we conclude that $G$ is the Green's function for $L$ in $\mathcal{B}$.

Next, we derive the following pointwise estimates for $G(x,y)$:
\begin{equation}			\label{1553thu}
G(x,y) \le C \abs{x-y}^{2-d},\quad  x \neq y \in \mathcal{B}.
\end{equation}
To see this, let $w=G(\,\cdot, y)$, and for $x_0 \in \mathcal{B}$ with $x_0 \neq y$, let $r=\abs{x_0-y}/3$.
Note that
\[
Lw=0 \;\text{ in }\;B_{3r}(x_0) \cap \mathcal{B}.
\]
By \cite[Theorem 4.2]{Krylov21}, employing a well-known iteration method, we establish that $w \in W^{2,p_0/2}(B_{2r}(x_0))$ and
\[
\sup_{B_{r}(x_0) \cap \mathcal{B}} \abs{w} \le C r^{-d} \norm{w}_{L^{1}(B_{2r}(x_0) \cap \mathcal{B})}.
\]
Through the pointwise estimates for $G_0(x, y)$ and \eqref{eq1410thu}, we find
\[
\norm{w}_{L^{1}(B_{2r}(x_0) \cap \mathcal{B})} \le C r^2+ Cr^{2+d(1-1/p_1)} \le C r^2.
\]
Consequently, by combining these results together, we deduce
\[
G(x_0,y) \le C r^{2-d}=C \abs{x_0-y}^{2-d},
\]
thereby confirming the estimate \eqref{1553thu}.

Now, by borrowing an idea from \cite{KL21}, we establish that
\begin{equation}		\label{symmetry}
G(x,y)=G^*(y,x),\quad x\neq y \in \Omega,
\end{equation}
where $G^*$ is the Green's function for $L^*$ in $\mathcal{B}$.
For any $x \in \mathcal{B}$ and $\epsilon>0$, set $f=-\abs{\mathcal{B} \cap B_\epsilon(x)}^{-1} \mathbbm{1}_{\mathcal{B} \cap B_\epsilon(x)}$, and let $v=G^*_\epsilon(\,\cdot, x)$ be a unique solution of the problem
\[
L^*v=f \;\text{ in }\; \mathcal{B},\quad v=0\; \text{ on }\;\partial \mathcal{B}.
\]
Then, for any $y \in \Omega$, we have
\begin{equation}			\label{eq0803mon}
G^*_\epsilon(y,x)=v(y)=-\int_{\mathcal{B}} f G(\,\cdot, y)=\fint_{B_\epsilon(x)\cap \mathcal{B}} G(\,\cdot, y).
\end{equation}
Especially when $y \neq x$ and $\epsilon<\frac12\abs{x-y}$, utilizing the pointwise estimate for $G(x,y)$ in \eqref{1553thu}, we derive
\[
\abs{G^*_\epsilon(y,x)} \le C\abs{x-y}^{2-d},
\]
which aligns with \cite[equation (2.27)]{HK20}.
Then by following the same argument in \cite{HK20}, we see that $G^*_\epsilon(\,\cdot, x)$ converges locally uniformly to $G^*(\,\cdot, x)$ in $\Omega \setminus \set{x}$.
By taking the limit $\epsilon \to 0$ in \eqref{eq0803mon}, we obtain \eqref{symmetry}.
The proof is completed.
\end{proof}
The subsequent theorem furnishes two-sided pointwise estimates for Green's function $G(x,y)$ in $\mathcal{B}$ for $x$, $y \in \overline{\mathcal{B}'}$.

\begin{theorem}		\label{thm_green_function}
Under the assumptions of Conditions \ref{cond1} and \ref{cond2}, there exists a positive constant $N$, dependent solely on $d$, $\lambda$, $\Lambda$, $\omega_{\mathbf A}$, $\diam \mathcal{B}$, and $\dist(\mathcal{B}', \partial\mathcal{B})$, such that the following inequalities hold for all $x \neq y \in \overline{\mathcal{B}'}$: 
\[
N^{-1} \abs{x-y}^{2-d} \le G(x,y)\le N \abs{x-y}^{2-d}.
\]
\end{theorem}
\begin{proof}
The upper bound is already given in \eqref{1553thu}.
We will utilize the Harnack inequality for nonnegative solutions of
\[
L^*u=D_{ij}(a^{ij}u)-D_i(b^i u)+cu=0,
\]
as established in \cite[Theorem 4.3]{GK24}.

For $x_0$, $y_0 \in \overline{\mathcal{B}'}$, let $r=\frac14 \abs{x_0-y_0}$.
Consider a smooth function $\eta \in C^\infty_c(B_{4r}(x_0))$ such that $\eta=1$ in $B_{3r}(x_0)$, $\norm{D\eta}_{L^\infty} \le 2/r$, and $\norm{D^2 \eta}_{L^\infty} \le 4/r^2$.
Then
\begin{align*}
1&=\eta(x_0)=\int_{\mathcal{B}} G(x_0, \cdot\,) L \eta= \int_{B_{4r}(x_0)\setminus B_{3r}(x_0)} G^*(\,\cdot,x_0) \left(a^{ij}D_{ij}\eta + b^i D_i \eta +c \eta \right)\\
&\le \sup_{B_{4r}\setminus B_{3r}} G^*(\,\cdot,x_0) \left( \frac{C}{r^2}\, \abs{B_{4r}} +\frac{C}{r} \,\norm{\vec b}_{L^{p_0}(B_{4r})} \abs{B_{4r}}^{1-1/p_0}+C \,\norm{\vec c}_{L^{p_0/2}(B_{4r})} \abs{B_{4r}}^{1-2/p_0}\right)\\
& \le C r^{d-2} \sup_{B_{4r}\setminus B_{3r}} G^*(\,\cdot, x_0),
\end{align*}
where we used H\"older's inequality and the fact that $p_0>d$.

We note that any two points in $B_{4r}(x_0)\setminus B_{3r}(x_0)$ can be connected by a chain of $k$ balls with radius $2r$ contained in $B_{5r}(x_0)\setminus B_{2r}(x_0)$, and $k$ does not exceed a number $k_0=k_0(d)$.
Therefore, applying the Harnack inequality for $G^*(\,\cdot, x_0)$ successively, we derive
\[
\sup_{B_{4r}\setminus B_{3r}} G^*(\,\cdot, x_0) \le C G^*(y_0,x_0)= C G(x_0,y_0).
\]
By putting these together, we deduce
\[
G(x_0,y_0) \ge C \abs{x_0-y_0}^{2-d}. \qedhere
\]
\end{proof}

\begin{remark}
In Theorem~\ref{thm_green_function}, it is worth noting that the upper bound applies to all $x$ and $y$ in $\mathcal{B}$, not just restricted to $x$ and $y$ both belonging to $\overline{\mathcal{B}'}$, such that
\[
G(x,y) \le N \abs{x-y}^{2-d},\quad x \neq y \in \mathcal{B}.
\]
Moreover, if $\vec b$ and $c$ have higher regularity, we also have pointwise estimates for $D_xG(x,y)$ and $D^2_x(x,y)$.
For instance, if $c \in L^{p_0}(\mathcal{B})$, then it follows that $G(\,\cdot, y) \in W^{2,p_0}_{\rm loc}(\mathcal{B} \setminus \set{y})$, and thus we have
\[
\abs{D_x G(x,y)} \le C \abs{x-y}^{1-d},\quad x \neq y \in \mathcal{B},
\]
where the constant $C$ also depends on $\norm{c}_{L^{p_0}(\mathcal{B})}$.
If $\vec b$ and $c$ are of Dini mean oscillation in $\mathcal{B}$, then we obtain
\[
\abs{D_x^2 G(x,y)} \le C \abs{x-y}^{-d},\quad x \neq y \in \mathcal{B},
\]
where the constant $C$ also depends on $\omega_{\vec b}$ and $\omega_c$, defined analogously to $\omega_{\mathbf A}$ in Condition \ref{cond2}. Refer to \cite[Theorem 1.5]{DEK18} for details.
\end{remark}

\section{Potential and Capacity}		\label{sec4}
Throughout this section, we assume Conditions \ref{cond1} and \ref{cond2} hold, with $c = 0$.

\begin{definition}		
A  function is said to as a potential in $\mathcal{B}$ if it satisfies the following conditions:
\begin{enumerate}
\item[(i)]
It is a nonnegative $L$-supersolution in $\mathcal{B}$.
\item[(ii)]
It is finite at every point in $\mathcal{B}$.
\item[(iii)]
It vanishes continuously on $\partial \mathcal{B}$.
\end{enumerate}
\end{definition}

\begin{definition}		\label{def2249sat}
For any extended real-valued function $u$ on $\mathcal{D}$, we define a lower semicontinuous regularization $\hat u$ of $u$ by
\[
\hat u(x)=\sup_{r>0} \left( \inf_{\mathcal{D} \cap B_r(x)} u\right).
\]
\end{definition}

\begin{remark}
We caution the reader that the preceding definition agrees with the one in some literature such as \cite{CC72}, but differs from the one appearing in some other literature such as \cite{Bauman85, Brelot, Helms}, where $\hat{u}$ is defined as
\[
\hat u(x)=\liminf_{y \to x} u(y).
\]
However, this distinction does not affect the conclusions of our results.
\end{remark}

\begin{lemma}			\label{lem1655sun}
The followings are true:
\begin{enumerate}[leftmargin=*]
\item
$\hat u \le u$.
\item
$\hat u$ is lower semicontinuous .
\item
If $v$ is a lower semicontinuous function satisfying $v \le u$, then $v \le \hat u$.
\item
If $u$ is lower semicontinuous, then $\hat u=u$.
\end{enumerate}
\end{lemma}
\begin{proof}
Refer to the Appendix for the proof.
\end{proof}

\begin{definition}			\label{def_reduced_function}
For $E \Subset \mathcal{B}$, we define
\[
u_E(x):=\inf \,\Set{v(x):  v \in \mathfrak{S}^+(\mathcal{B}),\; v\ge 0\text{ in }\mathcal{B},\; v \ge 1 \text{ in }E},\quad x \in \mathcal{B}.
\]
A lower semicontinuous regularization $\hat u_E$ is called the capacitary potential of $E$.
\end{definition}

\begin{lemma}				\label{lem1013sat}
The following properties hold:
\begin{enumerate}[leftmargin=*]
\item
$0\le \hat u_E \le u_E \le 1$.
\item
$u_E=1$ in $E$ and $\hat u_E =1$ in $\intr (E)$.
\item
$\hat u_E$ is a potential.
\item
$u_E=\hat u_E$ in $\mathcal{B}\setminus \overline{E}$ and
$L u_E=L \hat u_E=0$ in $\mathcal{B}\setminus \overline{E}$.
\item
If $x_0\in \partial E$, we have
\[
\liminf_{x\to x_0} \hat u_{E}(x) =\liminf_{x\to x_0,\; x\in \mathcal{B}\setminus  E} \hat u_E(x).
\]
\end{enumerate}
\end{lemma}
\begin{proof}
Refer to the Appendix for the proof.
\end{proof}

The following lemma is a restatement of \cite[Theorem 12.1]{Herve}.
We present the proof in a more accessible manner in the Appendix.
\begin{lemma}			\label{prop0800sun}
Let $u$ be a nonnegative $L$-supersolution in $\mathcal{B}$, and assume $u<+\infty$ in $\mathcal B$.
For an open set $\mathcal{D}\subset \mathcal{B}$, we define
\[
\mathscr{F}_{\mathcal D}=\set{v \in \mathfrak{S}^+(\mathcal{B}):  v\ge 0\,\text{ in }\,\mathcal{B},\; v - u \in \mathfrak{S}^+(\mathcal{D})}.
\]
For $x \in \mathcal{B}$, we then define
\[
\mathbb{P}_\mathcal{D} u(x)=\inf_{v \in \mathscr{F}_{\mathcal D}} v(x).
\]
The following properties hold:
\begin{enumerate}[leftmargin=*]
\item
$\mathbb{P}_\mathcal{D} u \in \mathfrak{S}^+(\mathcal{B})$ and $0 \le \mathbb{P}_\mathcal{D} u \le u$ in $\mathcal{B}$.
\item
$\mathbb{P}_\mathcal{D}u \in \mathscr{F}_{\mathcal D}$, i.e.,  $\mathbb{P}_\mathcal{D} u - u$ is an $L$-supersolution in $\mathcal{D}$.
\item
$\mathbb{P}_\mathcal{D} u$ is an $L$-solution in $\mathcal{B}\setminus \overline{\mathcal{D}}$.
\item
If $v \in \mathscr{F}_{\mathcal{D}}$, then $v-\mathbb{P}_\mathcal{D} u$ is a nonnegative $L$-supersolution in $\mathcal{B}$.
\end{enumerate}
\end{lemma}
\begin{proof}
Refer to the Appendix for the proof.
\end{proof}

\begin{remark}			\label{rmk1317tue}
It is evident from Lemma \ref{prop0800sun} (a) that $\mathbb{P}_\mathcal{D} u$ is a potential if $u$ is.
Moreover, as $u \in \mathscr{F}_{\mathcal{D}}$, Lemma \ref{prop0800sun} (d) implies that $u-\mathbb{P}_\mathcal{D} u$ is a nonnegative $L$-supersolution in $\mathcal{B}$.
Therefore, combined with Lemma \ref{prop0800sun} (b), we deduce that $u-\mathbb{P}_\mathcal{D} u$ is a nonnegative $L$-solution in $\mathcal{D}$.
\end{remark}

\begin{proposition}			\label{prop0800tue}
Let $u$ be a nonnegative $L$-supersolution in $\mathcal{B}$, and assume $u<+\infty$ in $\mathcal B$.
For each $x\in \mathcal{B}$, there exists a Radon measure $\mu^x$ on $\mathcal{B}$ such that
\begin{equation}			\label{eq0452fri}
\mu^x(\mathcal{D})=\mathbb{P}_{\mathcal{D}} u(x)
\end{equation}
for every open set $\mathcal{D}$ in $\mathcal{B}$.
\end{proposition}
\begin{proof}
For any open subset $\mathcal{D}$ of $\mathcal{B}$, consider $\mathbb{P}_{\mathcal D}$ as given in Lemma \ref{prop0800sun}.
The decomposition
\[
u=\mathbb{P}_\mathcal{D} u+(u-\mathbb{P}_\mathcal{D} u)=:\mathbb{P}_{\mathcal D}u+\mathbb{P}^{\perp}_{\mathcal D}u
\]
is characterized by the following conditions:
\begin{enumerate}
\item[(P1)]
$\mathbb{P}_{\mathcal D}u$ is a nonnegative $L$-supersolution in $\mathcal B$ and it is an $L$-solution in $\mathcal B \setminus \overline{\mathcal D}$.
\item[(P2)]
$\mathbb{P}^{\perp}_{\mathcal D}u$ is a nonnegative $L$-supersolution in $\mathcal B$ and it is an $L$-solution in $\mathcal D$.
\item[(P3)]
If two functions $v$ and $u-v$ are nonnegative $L$-supersolutions in $B$ and $v$ is an $L$-solution in $\mathcal D$, then $\mathbb{P}^{\perp}_{\mathcal D}u-v$ is a nonnegative supersolution in $\mathcal B$.
\end{enumerate}
Indeed, (P1) and (P2) follow from Lemma \ref{prop0800sun} and Remark \ref{rmk1317tue}.
To see (P3), observe that $u-v \in  \mathscr{F}_{\mathcal{D}}$, and thus by Lemma \ref{prop0800sun}(d), $(u-v)-\mathbb{P}_{\mathcal D}u=\mathbb{P}^{\perp}_{\mathcal D}u-v$ is a nonnegative $L$-supersolution in $\mathcal B$.

Let $\mathscr{A}$ be the collection of all open subsets of $\mathcal B$.
For fixed $x \in \mathcal{B}$, we define $\alpha: \mathscr{A} \to [0,+\infty)$ by
\[
\alpha(\mathcal D):=\mathbb{P}_{\mathcal D}u(x),
\]
with the convention that $\mathbb{P}_{\emptyset}u=0$.
Lemma \ref{lem1513tue} implies that $\alpha$ satisfies the following conditions:
\begin{enumerate}[leftmargin=*]
\item[(i)]
The relation $A \subset B$ implies $\alpha(A) \le \alpha(B)$.
\item[(ii)]
For any  $A$ and $B$ in $\mathscr{A}$, $\alpha(A \cup B) \le \alpha(A)+\alpha(B)$.
\item[(iii)]
The relation $A \cap B=\emptyset$ implies $\alpha(A \cup B)=\alpha(A)+\alpha(B)$.
\item[(iv)]
For every $\epsilon > 0$ and every $A \in \mathscr{A}$, there exist a compact set $K \subset A$ and an open set $U\supset A$ such that, for every $B \in \mathscr{A}$ satisfying the relation $K \subset B \subset U$, one has $\abs{\alpha(B)-\alpha(A)} \le \epsilon$.
\end{enumerate}

By \cite[Theorem 5, p.54]{Bourbaki}, there exists a unique Boreal measure $\mu^x$ on $\mathcal{B}$ such that \eqref{eq0452fri} holds for every open set $\mathcal D \subset \mathcal B$.
\end{proof}

The following lemma is a combined restatement of Lemmas 15.2 -- 15.4 in \cite{Herve}.
Recall that we denote
\[
\mathbb{P}^{\perp}_{\mathcal{D}}u := u-\mathbb{P}_{\mathcal{D}}u.
\]

\begin{lemma}				\label{lem1513tue}
The following statements hold:
\begin{enumerate}[leftmargin=*]
\item
If $\mathcal D_1$ and $\mathcal D_2$ are open subsets of $\mathcal{B}$ such that $\mathcal D_1 \subset \mathcal D_2$, then $\mathbb{P}_{\mathcal D_2}u-\mathbb{P}_{\mathcal D_1}u=\mathbb{P}^{\perp}_{\mathcal D_1}u-\mathbb{P}^{\perp}_{\mathcal D_2}u$ is a nonnegative $L$-supersolutions in $\mathcal{B}$ and it is an $L$-solution in $\mathcal{B} \setminus (\overline{\mathcal D_2 \setminus \mathcal D_1})$.
\item
If $\mathcal D_1$ and $\mathcal D_2$ are open subsets of $\mathcal{B}$, then we have
\begin{equation}			\label{eq1806mon}
\mathbb{P}_{\mathcal D_1 \cup \mathcal D_2}u+\mathbb{P}_{\mathcal D_1 \cap \mathcal D_2}u = \mathbb{P}_{\mathcal D_1}u + \mathbb{P}_{\mathcal D_2}u.
\end{equation}
\item
Let $\mathcal D$ is an open subset of $\mathcal B$, and let $\set{\mathcal D_n}$ be an increasing sequence of open subsets of $\mathcal B$ such that $\bigcup_n \mathcal D_n =\mathcal D$.
Then we have
\[
\mathbb{P}_{\mathcal D}u=\lim_{n\to \infty} \mathbb{P}_{\mathcal D_n}u=\sup_n \mathbb{P}_{\mathcal D_n}u.
\]
\end{enumerate}
\end{lemma}

\begin{proof}
\begin{enumerate}[leftmargin=*]
\item
Writing $u=\mathbb{P}_{\mathcal D_2}u+\mathbb{P}^{\perp}_{\mathcal D_2}u$ and noting  $\mathbb{P}^{\perp}_{\mathcal D_2}u$ is an $L$-solution in $\mathcal D_1$, we deduce from (P3) that $v:=\mathbb{P}^{\perp}_{\mathcal D_1}u-\mathbb{P}^{\perp}_{\mathcal D_2}u$ is a nonnegative $L$-supersolution in $\mathcal{B}$.
Since $\mathbb{P}^{\perp}_{\mathcal D_1}u-\mathbb{P}^{\perp}_{\mathcal D_2}u=\mathbb{P}_{\mathcal D_2}u-\mathbb{P}_{\mathcal D_1}u$, we find that $v$ is an $L$-solution in $\mathcal{B} \setminus \overline {\mathcal D_2}$ and also in $\mathcal D_1$.
Thus, the proof is complete if $\overline{\mathcal D_1} \subset \mathcal D_2$.

Otherwise, for $x_0 \in \partial \mathcal D_2 \setminus (\overline{\mathcal D_2\setminus \mathcal D_1})$, consider an open set $\mathcal{D}$ such that $x_0 \in \mathcal{D}$ and $\overline{\mathcal D} \cap (\mathcal D_2 \setminus \mathcal D_1)=\emptyset$.
The proof is complete if we show $\mathbb{P}_{D}v=0$, which implies $v$ is an $L$-solution in $\mathcal D$.

Note that $\mathbb{P}_{D}v$ is an $L$-solution in $\mathcal D_2$ as it is an $L$-solution in $\mathcal D_1$ and also in $\mathcal D_2 \setminus \mathcal D_1$.
Hence $\mathbb{P}^{\perp}_{\mathcal D_2}u+\mathbb{P}_{D}v$ is an $L$-solution in $D_2$.
On the other hand, since
\[
u=(\mathbb{P}^{\perp}_{\mathcal D_2}u+\mathbb{P}_{D}v)+(\mathbb{P}_{\mathcal D_2}u-\mathbb{P}_{D}v)=(\mathbb{P}^{\perp}_{\mathcal D_2}u+\mathbb{P}_{D}v)+(\mathbb{P}_{\mathcal D_1}u+\mathbb{P}^{\perp}_{D}v),
\]
and $\mathbb{P}_{\mathcal D_1}u+\mathbb{P}^{\perp}_{D}v$ is a nonnegative $L$-supersolution in $\mathcal{B}$, we deduce from (P3) that $\mathbb{P}^{\perp}_{\mathcal D_2}u-(\mathbb{P}^{\perp}_{\mathcal D_2}u+\mathbb{P}_{D}v)=-\mathbb{P}_{D}v \ge 0$ in $\mathcal{B}$.
Therefore, we conclude that $\mathbb{P}_{D}v=0$.

\item
It follows from (a) that $\mathbb{P}^{\perp}_{\mathcal D_1 \cap \mathcal D_2}u - \mathbb{P}^{\perp}_{\mathcal D_1}u$ is a nonnegative $L$-supersolution in $\mathcal{B}$ and it is an $L$-solution in $\mathcal D_2$.
Hence, $v:=\mathbb{P}^{\perp}_{\mathcal D_1 \cup \mathcal D_2}u +\mathbb{P}^{\perp}_{\mathcal D_1 \cap \mathcal D_2}u- \mathbb{P}^{\perp}_{\mathcal D_1}u$ is a nonnegative $L$-supersolution in $\mathcal{B}$ and it is an $L$-solution in $\mathcal D_2$.

Moreover, since
\[
u-v=\mathbb{P}_{\mathcal D_1 \cup \mathcal D_2}u - \mathbb{P}_{\mathcal D_1}u+ \mathbb{P}_{\mathcal D_1 \cap \mathcal D_2}u,
\]
$u-v$ is a nonnegative $L$-supersolution in $\mathcal B$.
Hence, it follows from (P3) that $\mathbb{P}^{\perp}_{\mathcal D_2}u-v$ is a nonnegative $L$-supersolution in $\mathcal{B}$.
In particular, we conclude that
\begin{equation}			\label{eq1731mon}
\mathbb{P}^{\perp}_{\mathcal D_1}u+\mathbb{P}^{\perp}_{\mathcal D_2}u\ge \mathbb{P}^{\perp}_{\mathcal D_1 \cup \mathcal D_2}u +\mathbb{P}^{\perp}_{\mathcal D_1 \cap \mathcal D_2}u.
\end{equation}

Similarly,  $w:=\mathbb{P}^{\perp}_{\mathcal D_1}u+\mathbb{P}^{\perp}_{\mathcal D_2}u-\mathbb{P}^{\perp}_{\mathcal D_1 \cap \mathcal D_2}u$ is a nonnegative $L$-supersolution in $\mathcal{B}$.
Note that $w$ is  an $L$-solution in $\mathcal D_1$, and also in $\mathcal D_2$.
Moreover, since
\[
u-w=\mathbb{P}_{\mathcal D_1}u + \mathbb{P}_{\mathcal D_1}u - \mathbb{P}_{\mathcal D_1 \cap \mathcal D_2}u,
\]
$u-w$ is a nonnegative $L$-supersolution in $\mathcal{B}$.
Again, we deduce from (P3) that $\mathbb{P}^{\perp}_{\mathcal D_1\cup \mathcal D_2}u-w \ge 0$  in $\mathcal{B}$.
This, along with \eqref{eq1731mon}, leads us to conclude that
\[
\mathbb{P}^{\perp}_{\mathcal D_1 \cup \mathcal D_2}u +\mathbb{P}^{\perp}_{\mathcal D_1 \cap \mathcal D_2}u = \mathbb{P}^{\perp}_{\mathcal D_1}u+\mathbb{P}^{\perp}_{\mathcal D_2}u.
\]
The proof is complete since the previous identity is equivalent to \eqref{eq1806mon}.

\item
It follows from (a) that $\set{\mathbb{P}_{\mathcal D_n}u}$ is an increasing sequence of nonnegative $L$-supersolutions in $\mathcal B$, and $\mathbb{P}_{\mathcal D_n}u \le \mathbb{P}_{\mathcal{D}}u$ for all $n$.
Therefore, $v:=\lim_{n\to \infty} \mathbb{P}_{\mathcal D_n}u$ exists and it is a nonnegative $L$-supersolution in $\mathcal{B}$.
Clearly, $v \le  \mathbb{P}_{\mathcal{D}}u$.

On the other hand, for any ball $B \Subset \mathcal{D}$, there is an integer $N$ such that  $\mathbb{P}^{\perp}_{\mathcal D_n}u$ is an $L$-solutions in $B$ for every $n \ge N$.
Since $\set{\mathbb{P}^{\perp}_{\mathcal D_n}u}_{n=N}^\infty$ is a decreasing sequence of $L$-solutions in $B$, its limit $w:=\lim_{n\to \infty} \mathbb{P}^{\perp}_{\mathcal D_n}u$ is an $L$-solution in $B$.
Since this holds for any $B\Subset \mathcal{D}$, we conclude that $w$ is an $L$-solution in $\mathcal{D}$.

Utilizing $v+w=u$ in $\mathcal D$, we infer from Lemma \ref{prop0800sun}(d) that $v-\mathbb{P}_{\mathcal D}u$ is a nonnegative $L$-supersolution in $\mathcal{B}$.
In particular, we have $v \ge  \mathbb{P}_{\mathcal D}u$.
Consequently, $v=\lim_{n\to \infty} \mathbb{P}_{\mathcal D_n}u=\mathbb{P}_{\mathcal{D}}u$. \qedhere
\end{enumerate}
\end{proof}

\begin{theorem}			\label{thm_capacity_measure}
Let $K$ be a compact subset of $\mathcal{B}$.
There exists a Borel measure $\mu=\mu_K$, referred to as the capacitary measure of $K$, supported in $\partial K$, such that
\[
\hat u_K(x)=\int_{K} G(x, y)\,d\mu(y), \quad \forall x \in \mathcal{B},
\]
where $G(x,y)$ is the Green's function for $L$ in $\mathcal{B}$.
We define the capacity of $K$ as
\[
\capacity(K)=\mu(K).
\]
\end{theorem}
\begin{proof}
By Lemma \ref{lem1013sat}, $\hat u_K$ is a potential in $\mathcal{B}$.
Hence, it follows from Proposition \ref{prop0800tue} that  there exists a family of measures $\set{\nu^x}_{x\in \mathcal{B}}$ such that for every open subset $\mathcal{D}$ of $\mathcal{B}$, the following holds:
\[
\mathbb{P}_{\mathcal D}\hat u_K(x)=\nu^x(\mathcal D).
\]

We first show that $\nu^x$ is supported in $\partial K$.
Let $y_0$ be an interior point of $K$.
Consider an open ball $B_r(y_0)$ such that $B_r(y_0) \Subset \intr(K)$.
By Lemma \ref{lem1013sat}, we find that $L \hat u_K= L1=0$ in $B_r(y_0)$.
Since $v=0$ satisfies the condition of Lemma \ref{prop0800sun}(d) with $u=\hat{u}_K$, we have $0-\mathbb{P}_{B_r(y_0)}\hat u_K \ge 0$.
This implies that $\mathbb{P}_{B_r(y_0)}\hat u_K=0$ since $\mathbb{P}_{B_r(y_0)}\hat u_K$ is nonnegative.
Similarly, if $y_0 \in \mathcal{B}\setminus K$, we choose an open ball $B_r(y_0)$ such that $B_r(y_0) \Subset \mathcal{B} \setminus K$. 
By Lemma \ref{lem1013sat}, we find that $L \hat u_K=0$ in $B_r(y_0)$, and thus, we arrive at the same conclusion that $\mathbb{P}_{B_r(y_0)}\hat u_K=0$.
Therefore, we infer that if $y_0 \not\in \partial K$, then there exists an open ball $B_r(y_0)$ such that $\nu^x(B_r(y_0))=0$.
This verifies that the support of $\nu^x$ lies on $\partial K$.

For $x_0$, $y \in \mathcal{B}$ with $x_0 \neq y$, consider a collections of balls $B_r(y)$ for $0<r<r_0$, where $r_0=\dist(y,\partial \mathcal{B})$.
By Lemma \ref{prop0800sun}, $\mathbb{P}_{B_r(y)}\hat{u}_{K}$ is a nonnegative $L$-supersolution the vanished on $\partial \mathcal{B}$.
Therefore, $\mathbb{P}_{B_r(y)}\hat{u}_{K}(x)>0$ for all $x \in \mathcal{B}$ or it is identically zero.
In the case when $\mathbb{P}_{B_r(y)}\hat{u}_{K} \not\equiv 0$, we define the function $v_r$ by
\[
v_r(x)=\frac{\mathbb{P}_{B_r(y)}\hat{u}_{K}(x)}{\mathbb{P}_{B_r(y)}\hat{u}_{K}(x_0)}.
\]

\begin{lemma}			\label{lem40}
Suppose $\nu^{x_0}(B_r(y))>0$ for all $r \in (0,r_0)$.
The family $\set{v_r}_{0<r<r_0}$ converges locally uniformly to $G(\,\cdot,y)/G(x_0,y)$ in $\mathcal{B}\setminus \set{x_0}$ as $r\to 0$.
\end{lemma}
\begin{proof}
As observed in Remark \ref{rmk1317tue}, $\mathbb{P}_{\mathcal D}\hat u_K$ is a potential in $\mathcal B$, implying that $v_r$ vanishes continuously on $\partial \mathcal{B}$.
Additionally, according to Lemma \ref{prop0800sun}, $v_r$ is an $L$-solution in $\mathcal{B}\setminus \overline B_r(y)$, and it is evident that $v_r(x_0)=1$.

Consider $\mathcal{B} \setminus \overline{B}_\epsilon(y)$ for fixed $\epsilon > 0$.
For $0<r<\epsilon$, each $v_r$ is an $L$-solution in $\mathcal{B} \setminus \overline B_\epsilon(y)$.
Since $v_r(x_0)=1$, the Harnack inequality (see \cite[Theorem 3.1]{Safonov}) ensures that $\set{v_r}_{0<r<\epsilon}$ is uniformly bounded in any compact set in $\mathcal{B} \setminus \overline B_\epsilon(y)$.
Therefore, by the $L^p$ theory (see \cite[Theorem 4.2]{Krylov21}), along with the standard diagonalization process, there exists a sequence $\set{v_{r_k}}_{k=1}^\infty$ with $r_k\to 0$ that converges weakly to a function $v$ in $W^{2,p_0}_{\rm loc}(\mathcal{B}\setminus \set{y})$.
It can be verified that $v$ satisfies the following properties:
\[
Lv =0\;\text{ in }\;\mathcal{B}\setminus \set{y},\quad v=0\;\text{ on }\;\partial\mathcal{B},\quad v(x_0)=1.
\]

We will show that $v=kG(\,\cdot, y)$, where $k$ is a nonnegative number.
Define  $k$ as the supremum of the set
\[
\set{ \alpha \in \mathbb{R}: v(x)-\alpha G(x,y) \ge 0\;\text{ for all }\;x\in \mathcal{B},\;x\neq y}.
\]
Since $v(x) \ge 0$ for all $x\neq y$, we deduce that $k \ge 0$.
Furthermore, from the pointwise estimate of $G(x,y)$ as given in Theorem \ref{thm_green_function}, we conclude that $k<+\infty$.
Therefore, $k$ is a nonnegative real number, and thus, we have $u-kG(\,\cdot,y) \ge 0$ in $\mathcal{B}\setminus \set{y}$.

Next, we will show that for any $\epsilon>0$, we have $u-(k+\epsilon)G(\,\cdot,y) \le 0$ in $\mathcal{B}\setminus \set{y}$.
Taking the limit as $\epsilon\to 0$, we conclude that $u=kG(\,\cdot,y)$.

Suppose to the contrary that there exists $x_1 \neq y$ such that
$u(x_1)-(k+\epsilon)G(x_1,y) >0$.
Note that $u-(k+\epsilon)G(\,\cdot,y)$ is an $L$-solution in $\mathcal{B}\setminus \set{y}$, and it vanishes on $\partial \mathcal{B}$.
Let $r_0$ be the minimum of $\abs{x_1-y}$ and  $\tfrac12\dist(y,\partial\mathcal{B})$.
The maximum principle implies that for any $r \in (0,r_0)$, we have
\begin{equation}			\label{eq0903mon}
0<\sup_{\partial B_r(y)} \left(u-(k+\epsilon)G(\,\cdot, y)\right) \le \sup_{\partial B_r(y)} \left(u-kG(\,\cdot, y)\right) -\inf_{\partial B_r(y)} \epsilon G(\,\cdot, y).
\end{equation}

We note that both $u-kG(\,\cdot, y)$ and $\epsilon G(\,\cdot, y)$ are nonnegative $L$-solutions in $\mathcal{B}\setminus \set{y}$.
Moreover, any two points on $\partial B_{r}(y)$ can be connected by a chain of $n$ balls with radius $r/2$ contained in $B_{2r}(y)\setminus B_{r/2}(y)$, where $n$ does not exceed a number $n_0=n_0(d)$.
Hence, applying the Harnack inequality as presented in \cite[Theorem 3.1]{Safonov} successively to $u-kG(\,\cdot, y)$, we derive
\begin{equation}			\label{eq0904mon}
\sup_{\partial B_r(y)} \left(u-kG(\,\cdot, y)\right) \le N\inf_{\partial B_r(y)} \left(u-kG(\,\cdot, y)\right),
\end{equation}
where $N$ is a constant depending only on $d$, $\lambda$, and $\Lambda$.
The same inequality holds for $\epsilon G(\,\cdot,y)$.
Therefore, we deduce from the inequalities \eqref{eq0903mon} and \eqref{eq0904mon} that
\[
0<N \inf_{\partial B_r(y)} \left(u-kG(\,\cdot, y)\right) - N^{-1} \sup_{\partial B_r(y)}  \epsilon G(\,\cdot, y),\quad \forall r \in (0, r_0).
\]
Consequently, we conclude that for any $x \in B_{r_0}(y)\setminus \set{y}$, we have
\[
0 \le u(x)-(k+\epsilon/N^2)G(x,y).
\]
As $u-(k+\epsilon/N^2)G(\,\cdot,y)=0$ on $\partial\mathcal{B}$, this inequality remains valid for all $x\in \mathcal{B}$ with $x \neq y$ by the maximum principle.
This is contradicts to the definition of $k$.

We have verified that $v=kG(\cdot, y)$.
Since $v(x_0)=1$, we conclude that $k=1/G(x_0,y)$.
Thus, the proof is complete
\end{proof}

By Lemma \ref{lem40}, we deduce that if $\nu^{x_0}(B_r(y))>0$ for all $r \in (0,r_0)$, then we have
\[
\lim_{r\to 0} v_r(x)=\lim_{r\to 0}\frac{\mathbb{P}_{B_r(y)}\hat{u}_{K}(x)}{\mathbb{P}_{B_r(y)}\hat{u}_{K}(x_0)}=\frac{G(x,y)}{G(x_0,y)},\quad \forall x \neq x_0,
\]
The above remains true for $x=x_0$, and thus we obtain
\[
\lim_{r\to 0}\frac{\nu^x(B_r(y))/G(x,y)}{\nu^{x_0}(B_r(y))/G(x_0,y)}=1,
\]
provided that $\nu^{x_0}(B_r(y)) >0$ for all $r \in (0, r_0)$.
Therefore, if we define a family of measures $\set{\mu^x}_{x\in \mathcal{B}}$ by
\[
d\mu^x(y)=\frac{1}{G(x,y)}d\nu^x(y),
\]
then $\mu^x$ is well-defined and $\mu^x=\mu^{x_0}$.
In other words, there exists a measure $\mu$ such that
\[
d\mu(y)=\frac{1}{G(x,y)}d\nu^x(y),\quad \forall x \in \mathcal{B}.
\]
Since $\nu^x(\mathcal B)=\mathbb{P}_{\mathcal{B}}\hat{u}_{K}(x)=\hat{u}_{K}(x)$, we have
\[
\hat{u}_K(x)=\nu^x(\mathcal B)=\int_{\mathcal B} 1 \,d\nu^x(y)=\int_{\mathcal B} G(x,y)\,d\mu(y).
\]

Finally, it is clear that $\mu$ is supported in $\partial K$.
This concludes the proof of the theorem.
\end{proof}

We note that the capacitary measure $\mu$ defined in the previous theorem differs from the one introduced in \cite[Theorem 1.9]{Bauman85}.
Our definition aligns with the "equilibrium measure" introduced in \cite{GW82}, particularly when $L=\Delta$, the Laplace operator.
For further discussion, refer to Remark \ref{rmk0918thu} below.

\begin{remark}	\label{rmk0918thu}
Consider the case when $L=\Delta$, the Laplace operator.
Let $K$ be a compact set in $\mathcal{B}$.
Define $\tilde{u}_K$ as the minimizer of the following problem:
\[
\inf_{v \in \mathscr{A}_K} \int_\mathcal{B} \,\abs{\nabla v}^2\,dx,
\]
where $\mathscr{A}_K$ denotes the set of functions $v \in W^{1,2}_0(\mathcal{B})$ satisfying $v \ge 1$ on $K$.
It is well known that there exists a measure $\nu$ supported in $K$ such that
\begin{equation}			\label{eq1152tue}
\tilde u_K(x)=\int_K G(x,y)\,d\nu(y).
\end{equation}
Moreover, $\tilde{u}_K$ satisfies the property:
\[
\int_{\mathcal{B}} \nabla \tilde u_K \cdot \nabla \phi\,dx= \int_K \phi \,d \nu,\quad \forall \phi \in C^\infty_c(\mathcal{B}).
\]
These results can be found, for instance, in \cite[p. 322]{GW82}.
In particular, we have
\begin{equation}			\label{eq2029tue}
\int_{\mathcal{B}} \abs{\nabla \tilde u_K}^2\,dx =  \int_K \tilde u_K d\nu.
\end{equation}
By symmetry of the Green's function of the Laplace operator, Fubini's theorem, Theorem \ref{thm_capacity_measure}, and \eqref{eq1152tue}, we obtain:
\begin{equation}			\label{eq2030tue}
\int_K \hat u_K\, d\nu =\int_K \int_K G(x,y)\,d\mu(y) d\nu(x)=\int_K \int_K G(x,y)\,d\nu(y)d\mu(x)= \int_K \tilde u_K\, d\mu.
\end{equation}
It is well known that if $\partial K$ possesses certain smoothness properties such as the uniform cone condition, then $\hat{u}_K=\tilde{u}_K=1$ on $K$.
Consequently, from \eqref{eq2029tue}, \eqref{eq2030tue}, and Theorem \ref{thm_capacity_measure}, we deduce:
\[
\int_{\mathcal{B}} \abs{\nabla \tilde u_K}^2\,dx=\nu(K)=\mu(K).
\]
This confirms the equivalence between our definition of capacity and the traditional notion.
\end{remark}

The definition in Theorem \ref{thm_capacity_measure} makes it clear that the notion of capacity is associated with a specific operator $L$.
Theorem \ref{thm_equivalence} below asserts that the capacity associated with $L$ is comparable to the capacity associated with the Laplace operator provided that Conditions \ref{cond1} and \ref{cond2} hold, with $c = 0$.
We need the following lemmas for the proof of Theorem \ref{thm_equivalence}.

\begin{lemma}		\label{lem2020sat}
Let $K$ be a compact subset of $\mathcal{B}$ and let $\mathfrak{M}_K$ be the set of all Borel measures $\nu$ supported in $K$ satisfying $\int_K G(x,y)\,d\nu(y) \le 1$ for every $x \in \mathcal{B}$.
Then
\[
\hat u_{K}(x)=\sup_{\nu \in \mathfrak{M}_K} \int_K G(x,y)\,d\nu(y).
\]
\end{lemma}
\begin{proof}
Let  $g_\nu(x):=\int_K G(x,y)\,d\nu(y)$.
By Theorem \ref{thm_capacity_measure}, we have
\[
\hat{u}_K(x) \le \sup_{\nu \in \mathfrak{M}_K} g_{\nu}(x)=\sup_{\nu \in \mathfrak{M}_K} \int_K G(x,y)\,d\nu(y).
\]
Next, let $v$ be any function from $\mathfrak{S}^+(\mathcal{B})$ satisfying $v\ge 0$ in $\mathcal{B}$ and $v \ge 1$ in $K$.
Then, for any $\nu \in \mathfrak{M}_K$, we have $g_\nu \le 1 \le v$ in $K$.
In particular, we have $v \geq 1 \geq g_{\nu}$ on $\partial K$ and this implies that $v \geq g_{\nu}$ on $\partial(\mathcal{B}\setminus K)$.
Since $v$ is $L$-supersolution in $\mathcal{B}$ and $g_{\nu}$ is $L$-solution in $\mathcal{B}\setminus K$, the comparison principle shows that $v \ge g_{\nu}$ in $\mathcal{B} \setminus K$.
Hence, it follows that $v \ge g_{\nu}$ in $\mathcal{B}$.
By taking the infimum over $v$, we deduce that
\[
u_K(x) \ge g_{\nu}(x)=\int_K G(x,y)\,d\nu(y),\quad \forall x \in \mathcal{B}.
\]
Moreover, we have
\begin{align*}
\hat{u}_K (x) &= \sup_{r> 0} \left(\inf_{z \in B_r(x) \cap \mathcal{B}}u_K(z) \right) \geq \lim_{r \to 0} \bigg(\inf_{z \in B_r(x)\cap \mathcal{B}}\int_{K}G(z,y)\,d\nu(y)\bigg)\\ &\geq \int_K\liminf_{z \to x}G(z,y)\,d\nu(y) =\int_K G(x,y)\,d\nu(y),
\end{align*}
where we used Fatou's lemma.
Taking the supremum over $\nu\in \mathfrak{M}_K$, the proof is complete.
\end{proof}

\begin{lemma}				\label{lem0956fri}
For any compact set $K \subset \mathcal{B}''$ and any point $X^o \in \mathcal{B}\setminus \mathcal{B}'$, we have
\[
N^{-1}  \dist(\mathcal{B}'', X^o)^{d-2} \hat u_K(X^o) \le \capacity(K)  \le N (\diam \mathcal{B})^{d-2}  \hat u_K(X^o),
\]
where $N$ is the constant in Theorem \ref{thm_green_function}.
\end{lemma}

\begin{proof}
By Theorem \ref{thm_capacity_measure}, we have
\begin{equation*}			
\hat u_K(X^o)=\int_{K}G(X^o, y)\,d\mu(y),
\end{equation*}
where we have used the fact that the capacitary measure $\mu$ is supported in $K$.
The lemma follows immediately from the two-sided estimates for the Green's function in Theorem \ref{thm_green_function}, since for $y \in K$, we have $\dist(\mathcal{B}'', X^o) \le \abs{X^o-y} \le \diam \mathcal{B}$.
\end{proof}

\begin{theorem}			\label{thm_equivalence}
Let $K$ be any compact subset of $\mathcal{B}''$.
Denote by $\capacity^0(K)$ the capacity of $K$ associated with the Laplace operator.
There exists a constant $N>0$, depending only on $d$, $\lambda$, $\Lambda$, $\omega_A$, $\diam(\mathcal{B})$, $\dist(\mathcal{B}'' ,\partial \mathcal{B}')$, and $\dist(\mathcal{B}' ,\partial \mathcal{B})$, such that
\[
N^{-1} \capacity^0(K) \le \capacity(K) \le  N \capacity^0(K).
\]
\end{theorem}

\begin{proof}
Let $G^0(x,y)$ be Green's function for the Laplace operator in $\mathcal{B}$, and denote by $\hat u_K^0$ and $\mu^0$, respectively, the capacitary potential of $K$ and the capacity of $K$ for the Laplace operator.

By Theorem \ref{thm_green_function}, there is a constant $\alpha>0$ such that
\begin{equation}		\label{greens}
\alpha G^0(x,y) \le G(x,y) \le \alpha^{-1} G^0(x,y), \quad \forall x \neq y \in \overline{\mathcal{B}'}.
\end{equation}
Let $\mu$ be any Borel measure supported in $K$ satisfying
\begin{equation}		\label{eq1844thu}
\int_K G^0(x, y)\, d\mu(y) \leq 1,\quad \forall x \in \mathcal{B}.
\end{equation}
Setting $\tilde{\mu}:= \alpha{\mu}$, and utilizing the inequalities \eqref{greens}, we obtain
\begin{equation}			\label{eq1747thu}
\int_{K} G(x, y)\,d\tilde\mu(y) \le \int_K G^0(x,y)\, d\mu(y) \le 1,\quad \forall x \in \overline{\mathcal{B}'}.
\end{equation}
Note that $u(x):=\int_K G(x, y)\, d\tilde\mu(y)$ satisfies $L u=0$ in $\mathcal{B} \setminus K$.
Since $u=0$ on $\partial \mathcal{B}$ and $u \le 1$ on $\partial K$, it follows from the comparison principle that $u \le 1$ in $\mathcal{B} \setminus K$.
This combined with \eqref{eq1747thu} imply that
\[
\int_K G(x,y) d \tilde\mu(y) \le 1,\quad \forall x \in \mathcal{B}.
\]
Therefore, by Lemma~\ref{lem2020sat}, we obtain
\[
\int_K G(x,y) d \tilde\mu(y) \le \hat u_K(x),\quad \forall x \in \mathcal{B}.
\]
Since the measure $\mu$ is supported in $K \subset \mathcal{B}'$, it follows from the inequalities \eqref{greens} that for $x \in \overline{\mathcal{B}'}$, we have
\[
\int_{K} G^0(x, y)\,d\mu(y) \leq \alpha^{-1} \int_K G(x,y)\, d\mu(y) = \alpha^{-2} \int_K G(x ,y)\,d\tilde{\mu}(y) \le \alpha^{-2}  \hat u_K(x).
\]
By taking the supremum over $\mu$ satisfying \eqref{eq1844thu} and by using lemma \ref{lem2020sat} again, we deduce
\[
\hat u^0_K(x) \le \alpha^{-2} \hat u_K(x) \le \alpha^{-2} u_K(x),\quad \forall x \in \overline{\mathcal{B}'}.
\]
By taking $X^o \in \partial\mathcal{B}'$, and using Lemma \ref{lem0956fri}, we deduce that  $\capacity^0(K) \le N \capacity(K)$ for some constant $N$.
By interchanging the role of $G^0(x,y)$ and $G(x,y)$, we can also conclude that $\capacity(K) \le N \capacity^0(K)$.
\end{proof}

\begin{lemma}		\label{lem1650thu}
Let $N$ be the constant in Theorem \ref{thm_green_function}.
We have
\[
\capacity(K) \le 3^{d-2} N  r^{d-2}
\]
whenever $K\subset \overline B_r(x_0)$ and $B_{2r}(x_0) \subset \mathcal{B}'$.
\end{lemma}
\begin{proof}
Observe that
\[
\sup\Set{\,\abs{x-y}: x \in \partial B_{2r}(x_0), y \in K  \,} \le 3r.
\]
Therefore, it follows from Theorem \ref{thm_green_function}, Lemma \ref{lem1013sat}, and Theorem \ref{thm_capacity_measure} that for $x \in \partial B_{2r}(x_0)$, we have
\[
N^{-1}3^{2-d} r^{2-d} \mu(K)\le \int_K G(x,y)\,d\mu(y) =\hat u_K(x) \le 1. \qedhere
\]
\end{proof}

\section{Wiener test}			\label{sec5}
Recall that we assume $\mathcal{B}'' \subset \mathcal{B}' \subset \mathcal{B}$ are concentric open balls.
Let $r_0$ be a fixed number such that $0<r_0<\frac12 \dist(\mathcal{B}'', \partial \mathcal{B}')$.

\begin{theorem}		\label{thm_wiener}
Assume Conditions \ref{cond1} and \ref{cond2} hold, with $c = 0$.
Let $\Omega \subset \mathcal{B}''$ be an open set.
The point $x_0 \in \partial \Omega$ is a regular point if and only if
\begin{equation}			\label{eq1724thu}
\int_0^{r_0} \frac{\capacity(E_r)}{r^{d-1}}dr=+\infty,\; \text{ where }\,E_r:=\overline B_r(x_0)\setminus \Omega.
\end{equation}
\end{theorem}
\begin{proof}
(Necessity)
We will show that $x_0$ is not a regular point if
\begin{equation}			\label{eq0801thu}
\int_0^{r_0} \frac{\capacity(E_r)}{r^{d-1}}dr<+\infty.
\end{equation}
Let $\hat u_{E_r}$ and $\mu_{r}$ denote the capacitary potential and the capacitary measure of the compact set $E_r$, respectively, where $0<r<r_0/2$.
By Theorem \ref{thm_capacity_measure}, we have
\begin{equation}			\label{eq1353thu}
\hat u_{E_r}(x)=\int_{E_r} G(x, y)\,d\mu_r(y).
\end{equation}
Note that for $0<\rho<r$, we have
\[
\int_{E_\rho}G(x_0, y)\,d\mu_r(y)  \le \int_{E_r}G(x_0, y)\,d\mu_r(y)=\hat u_{E_r}(x_0) \le 1.
\]
By the dominated convergence theorem, we have
\[
\lim_{\rho \to 0} \int_{E_\rho}G(x_0, y)\,d\mu_r(y) = \int_{\set{x_0}} G(x_0,y)\,d\mu_r(y) \le 1.
\]
Since $G(x_0,x_0)=+\infty$, it follows $\mu_r(\set{x_0})=0$, which yields
\begin{equation}			\label{eq1445thu}
\int_{\set{x_0}} G(x_0,y)\,d\mu_r(y)=0.
\end{equation}
Hence, by partitioning $E_r$ into annular regions and using the Green's function estimates in Theorem~\ref{thm_green_function}, we derive from \eqref{eq1353thu} and \eqref{eq1445thu} the following:
\begin{align}
				\nonumber
\hat u_{E_r}(x_0) &= \sum_{k=0}^\infty \int_{E_{2^{-k}r}\setminus E_{2^{-k-1}r}} G(x_0, y)\,d\mu_r(y)+ \int_{\set{x_0}} G(x_0,y)\,d\mu_r(y) \\
& \nonumber \lesssim \sum_{k=0}^\infty r^{2-d} 2^{-k(2-d)} \mu_r(E_{2^{-k}r}) \lesssim  \frac{\mu_r(E_r)}{r^{d-2}} + \sum_{k=1}^{\infty}\int_{2^{-k}r}^{2^{-k+1}r}\frac{\mu_r(E_t)}{t^{d-1}}\,dt\\
&\lesssim \frac{\mu_r(E_r)}{r^{d-2}} + \int_0^r \frac{\mu_r(E_t)}{t^{d-1}}\,dt.
				\label{eq1725sat}
\end{align}

We will show that there exists a constant $C$ independent of $t$ such that
\begin{equation}			\label{eq1727sat}
\mu_r(E_t) \le C \mu_t(E_t)=C \capacity(E_t), \quad \forall t \in (0,r].
\end{equation}
Then the expression on the right-hand side of \eqref{eq1725sat}  tends to zero as $r \to 0$ due to the hypothesis \eqref{eq0801thu} and the following observation:
\[
\frac{\mu_r(E_r)}{r^{d-2}} \lesssim \int_{r}^{2r} \frac{\mu_r(E_t)}{t^{d-1}}\,dt.
\]
We choose $r$ to be sufficiently small such that $\hat u_{E_r}(x_0)<1$.
Consider the function
\[
f(x)=\left(1-\abs{x-x_0}/r\right)_+,
\]
and let $u=\underline{H}_f$ represent the lower Perron's solution to the Dirichlet problem $Lu=0$ in $\Omega$ with $u=f$ on $\partial\Omega$.
From Definition~\ref{def_reduced_function} and the comparison principle, it follows that $u \le u_{E_r}$ in $\Omega$.
Since $u_{E_r}=\hat u_{E_r}$ in $\Omega$, it implies $\underline{H}_f=u \le \hat u_{E_r}$ in $\Omega$.
However, this contradicts \eqref{eq1212thu} since
\[
\liminf_{x\to x_0,\,x\in\Omega}\, \underline{H}_f(x) \le \liminf_{x\to x_0,\,x\in\Omega}\, \hat u_{E_r}(x) = \liminf_{x\to x_0}\, \hat u_{E_r}(x)= \hat u_{E_r}(x_0) < 1=f(x_0),
\]
where we applied Lemmas \ref{lem1013sat} and \ref{lem2334sun}.
Therefore, $x_0$ is not a regular point.

It only remains to prove \eqref{eq1727sat}.
For $0<t \le r$, let $\nu$ be the measure defined by
\[
\nu(E)=\mu_r(E_t \cap E).
\]
It is clear that $\nu$ is supported in $E_t$.
Moreover, for every $x \in \mathcal{B}$, we have
\[
\int_{E_t} G(x,y)\,d\nu(y) =\int_{E_t} G(x,y)\,d\mu_r(y) \le \int_{E_r} G(x,y)\,d\mu_r(y) \le 1.
\]
Therefore, by Lemma \ref{lem2020sat}, we deduce that
\[
\int_{E_t} G(x,y)\,d\mu_r(y)=\int_{E_t} G(x,y)\,d\nu(y) \le \hat u_{E_t}(x)=\int_{E_t} G(x,y)\,d\mu_t(y),\quad \forall x \in \mathcal{B}.
\]
By taking $x=X^o \in \partial \mathcal{B}'$ in the previous inequality, and using Theorem~\ref{thm_green_function}, we obtain \eqref{eq1727sat}.
This completes the proof for necessity.\\

\noindent
(Sufficiency)
It suffices to demonstrate the existence of a ``barrier'' function $u$ satisfying the following conditions:
\begin{enumerate}[leftmargin=*]
\item[(i)]
$u$ is an $L$-solution in $\Omega$.
\item[(ii)]
For any $\epsilon>0$, there exists $\delta>0$ such that $u \ge \delta$ on $\partial\Omega \setminus B_\epsilon(x_0)$.
\item[(iii)]
$\lim_{x\to x_0,\, x\in \Omega} u(x)=0$.
\end{enumerate}
Let $\rho \in (0,r_0)$ be a fixed yet arbitrarily small number.
The subsequent lemma is a minor modification of \cite[Lemma 2.3]{GW82}.
\begin{lemma}			\label{lem2119thu}
There exist constants $\alpha_0 \in (0,\frac12)$ and $\kappa>0$ such that for all $\alpha \in (0,\alpha_0)$ and $r\in (0, \alpha \rho]$ we have
\[
\sup_{\Omega \cap B_r(x_0)}\left (1- u_{E_{\rho}}\right) \le \exp\left\{ -\frac{\kappa \alpha^{d-1}}{1-\alpha}\int_r^{\alpha \rho} \frac{\capacity(E_t)}{t^{d-1}}\,dt\right\}.
\]
\end{lemma}
\begin{proof}
In the proof, we express
\begin{equation}			\label{eq0704sun}
v(x)=1-u_{E_\rho}(x)=\sup\Set{w(x):w \in \mathfrak{S}^-(\mathcal{B}),\; w \le 1\text{ in }\mathcal{B},\; w \le 0 \text{ in }\overline B_{\rho}(x_0)\setminus \Omega}.
\end{equation}
Here, we used Definition \ref{def_reduced_function}.
Note that $0\le v \le 1$.

Let $0<\alpha<\frac12$ and  $0<t \le \rho$.
It follows from Theorems \ref{thm_green_function} and \ref{thm_capacity_measure} that for $x \in \partial B_t(x_0)$, we have
\[
1- \hat u_{E_{\alpha t}}(x) =1-\int_{E_{\alpha t}} G(x,y)\,d\mu_{\alpha t} \ge 1-N2^{d-2} t^{2-d} \capacity(E_{\alpha t}).
\]
Applying Lemma \ref{lem1650thu}, we observe that
\[
1-N2^{d-2}t^{2-d}\capacity(E_{\alpha t}) \ge  1-6^{d-2}N^2 \alpha^{d-2}>0
\]
if $0<\alpha<\frac16 N^{-2/(d-2)}$.
Consequently, for $x\in \Omega \cap \partial B_{t}(x_0)$, we have
\begin{equation}			\label{eq2108fri}
\left(1-\hat u_{E_{\alpha t}}(x)\right) \sup_{\Omega \cap B_t(x_0)} v \ge \left(1- 2^{d-2}N t^{2-d} \capacity(E_{\alpha t})\right) v(x).
\end{equation}
Note that \eqref{eq2108fri} is also valid on $\partial\Omega \cap \overline B_{t}(x_0)$ due to \eqref{eq0704sun}, and by the comparison principle, inequality \eqref{eq2108fri} holds for all $x \in \Omega \cap B_t(x_0)$.

On the other hand, by Theorems \ref{thm_green_function} and \ref{thm_capacity_measure}, for $x \in \partial B_{2\alpha t}(x_0)$, we have
\begin{equation}			\label{eq2109fri}
1\ge \hat u_{E_{\alpha t}}(x) \ge N^{-1} (3\alpha t)^{2-d} \capacity(E_{\alpha t}).
\end{equation}

Then, for $x \in \Omega\cap \partial B_{2\alpha t}$, inequalities \eqref{eq2108fri} and \eqref{eq2109fri} imply
\begin{align*}			
v(x) &\le \sup_{\Omega \cap B_t(x_0)}v \;\;\left( \frac{1- N^{-1}3^{2-d}(\alpha t)^{2-d} \capacity(E_{\alpha t})}{1- N2^{d-2} t^{2-d} \capacity(E_{\alpha t})}\right)\\
&\le \sup_{\Omega \cap B_t(x_0)}v \;\; \left(1- \frac{N^{-1}3^{2-d}}{2}(\alpha t)^{2-d} \capacity(E_{\alpha t})\right),
\end{align*}
with the last inequality valid for $0<\alpha \le \frac16 (2N^2)^{-1/(d-2)}$.
Again, by \eqref{eq0704sun} and the comparison principle,
we observe that the previous inequality remain valid for every $x \in \Omega \cap B_{2\alpha t}(x_0)$.
Using the notation
\[
M(r):=\sup_{\Omega \cap B_r(x_0)} v,
\]
we have established the existence of positive numbers $\alpha_0$ and $\kappa$ such that for $t \in (0, \rho]$ and  $\alpha \in (0, \alpha_0)$, the following inequality holds:
\[
M(\alpha t) \le M(t) \left\{1-\kappa(\alpha t)^{2-d} \capacity(E_{\alpha t}) \right\}.
\]
Utilizing the inequality $\ln (1-x) \le -x$ for $x <1$, we deduce
\[
M(\alpha t) \le M(t) \exp\left\{-\kappa (\alpha t)^{2-d} \capacity(E_{\alpha t}) \right\}.
\]
By iterating, we deduce that
\[
M(\alpha^j \rho) \le M(\rho) \exp\left\{-\kappa \sum_{i=1}^j (\alpha^i \rho)^{2-d} \capacity(E_{\alpha^i \rho}) \right\},\quad j=1,2,\ldots.
\]
Note that for $r \in (\alpha^{j+1} \rho, \alpha^j \rho]$, we have
\begin{multline*}
\int_r^{\alpha \rho} \frac{\capacity(E_t)}{t^{d-1}}\,dt \le \int_{\alpha^{j+1} \rho}^{\alpha \rho} \frac{\capacity(E_t)}{t^{d-1}}\,dt =\sum_{i=1}^{j}\int_{\alpha^{i+1} \rho}^{\alpha^i \rho} \frac{\capacity(E_t)}{t^{d-1}}\,dt\\
\le \sum_{i=1}^{j} \frac{\capacity(E_{\alpha^i \rho})}{(\alpha^{i+1}\rho)^{d-1}}(\alpha^i\rho-\alpha^{i+1}\rho)=\alpha^{1-d}(1-\alpha)\sum_{i=1}^{j} (\alpha^{i}\rho)^{2-d}\capacity(E_{\alpha^i \rho}).
\end{multline*}
For $r \in (0,\alpha \rho]$, choose $j$ such that  $r \in (\alpha^{j+1} \rho, \alpha^j \rho]$, then we have
\[
M(r) \le M(\alpha^j \rho) \le M(\rho)\exp\left\{ -\frac{\kappa \alpha^{d-1}}{1-\alpha}\int_r^{\alpha \rho} \frac{\capacity(E_t)}{t^{d-1}}\,dt\right\}.\qedhere
\]
\end{proof}

By Lemma~\ref{lem2119thu}, we obtain the following inequality for any $\alpha \in (0, \alpha_0)$:
\begin{equation}			\label{eq1332thu}
1-u_{E_{\rho}}(x)\le \exp\left(-\frac{\kappa \alpha^{d-1}}{1-\alpha}\int_{\abs{x-x_0}}^{\alpha\rho}\frac{\capacity(E_t)}{t^{d-1}}\,dt\right),\quad  \forall x \in \Omega \cap B_{\alpha \rho}(x_0).
\end{equation}
On the other hand, by Theorems \ref{thm_green_function}, \ref{thm_capacity_measure}, Lemmas \ref{lem1013sat}, and \ref{lem1650thu}, for $x \in \partial B_{\rho/\alpha}(x_0)$, we have
\[
u_{E_{\rho}}(x) = \hat u_{E_{\rho}}(x) \le N \left\{\left(\frac{1}{\alpha}-1\right)\rho\right\}^{2-d} \capacity(E_{\rho}) \le N^2 \left(\frac{3\alpha}{1-\alpha}\right)^{d-2}.
\]
We choose $\alpha \in (0,\alpha_0)$ such that the right-hand side of the above inequality becomes less than $\frac12$, i.e., we have
\begin{equation}			\label{eq1723thu}
u_{E_{\rho}}(x) <1/2,\quad \forall x \in \partial B_{\rho/\alpha}(x_0).
\end{equation}

Now, consider the function
\[
f(x)=\abs{x-x_0},
\]
and let $u=\underline H_f$ be the lower Perron's solution to the Dirichlet problem $Lu=0$ in $\Omega$, with $u=f$ on $\partial\Omega$.
It is evident that $0 \le u \le \diam (\Omega)$ by the definition of $\underline H_f$.
Therefore, based on \eqref{eq1723thu}, we have
\[
u(x) \le
\begin{cases}
2\diam(\Omega) \left(1-u_{E_{\rho}}(x)\right)& \text{in }\;\Omega \cap \partial B_{\rho/\alpha}(x_0),\\
\rho/\alpha &  \text{in }\; \partial\Omega \cap \overline B_{\rho/\alpha}(x_0).
\end{cases}
\]
Note that the function $w$, defined by
\[
w(x)=2\diam(\Omega) \left(1-u_{E_{\rho}}(x)\right)+\rho/\alpha,
\]
satisfies $Lw =0$ in $\Omega \cap B_{\rho/\alpha}(x_0)$ and $w \ge \rho/\alpha$ in $\Omega$.
Therefore, the comparison principle applied to $u$ and $w$ on $\Omega \cap B_{\rho/\alpha}(x_0)$, combined with \eqref{eq1332thu}, implies
\[
u(x) \le \frac{\rho}{\alpha}+2\diam(\Omega) \exp\left(-\frac{\kappa \alpha^{d-1}}{1-\alpha} \int_{\abs{x-x_0}}^{\alpha\rho}\frac{\capacity(E_t)}{t^{d-1}}\,dt\right),\quad \forall x \in \Omega \cap B_{\alpha \rho}(x_0).
\]
Subsequently,  \eqref{eq1724thu} indicates
\[
\limsup_{x \to x_0,\,x\in \Omega} u(x) \le \rho/\alpha.
\]
As $\rho>0$ can be arbitrarily small, $u$ is constructed independently of $\rho$ and $\alpha$,  and $u \ge 0$, we conclude:
\[
\lim_{x \to x_0,\,x\in \Omega} u(x) =0.
\]

Let $R=\diam \Omega$. For any $\epsilon \in (0,R)$, let $v$ be the solution of the following problem:
\[
Lv=0\;\text{ in }\; B_R(x_0)\setminus \overline B_{\epsilon/2}(x_0),\quad v=0\;\text{ on }\;\partial B_{\epsilon/2}(x_0),\quad v=\epsilon/2\;\text{ on }\;\partial B_{R}(x_0).
\]
Extend $v$ to be zero in $\overline B_{\epsilon/2}(x_0)$.
Notice that $v$ is a continuous $L$-subsolution in $B_R(x_0)$ satisfying $0\le v \le \epsilon/2$.
Since $v \le f$ on $\partial \Omega$, it follows from Definition \ref{def_perron} that $v\le \underline H_f=u$ in $\Omega$.
The argument in the proof Proposition \ref{lem1012thu}  reveals that  $v \ge \delta$ in $B_R(x_0) \setminus B_{\epsilon}(x_0)$ for some $\delta>0$, implying that $u \ge \delta$ in $\Omega\setminus B_{\epsilon}(x_0)$.
Thus, we have confirmed that $u$ satisfies all the conditions (i) -- (iii) to be a "barrier."
This completes the proof of Theorem \ref{thm_wiener}.
\end{proof}

\begin{theorem}		\label{cor0800sat}
Assume Conditions \ref{cond1} and \ref{cond2} hold.
Let $\Omega \subset \bR^d$ be a bounded open set.
A point $x_0 \in \partial \Omega$ is a regular point for $L$ if and only if $x_0$ is a regular point for the Laplace operator.
\end{theorem}
\begin{proof}
In the case when $c = 0$, the theorem follows directly from Theorems \ref{thm_wiener} and \ref{thm_equivalence}.
Then, by the discussion at the end of Section \ref{sec2.1}, we can remove the assumption $c = 0$.
\end{proof}

\begin{definition}
                    \label{def_reg}
We define a domain $\Omega$ as a regular domain if every point on its boundary $\partial \Omega$ is a regular point with respect to the Laplace's operator.
\end{definition}

\begin{theorem}		\label{thm0802sat}
Under the assumptions that Conditions \ref{cond1} and \ref{cond2} are satisfied, consider a bounded regular domain $\Omega \subset \bR^d$.
For $f \in C(\partial\Omega)$, the Dirichlet problem,
\begin{equation}			\label{eq0751sat}
Lu=0 \;\text{ in }\;\Omega, \quad u=f \;\text{ on }\;\partial\Omega,
\end{equation}
possesses a unique solution in $W^{2,p_0/2}_{\rm loc}(\Omega)\cap C(\overline\Omega)$.
\end{theorem}
\begin{proof}
Let $\zeta$ be as in Proposition \ref{lem1012thu}, and consider the operator $\tilde L$ defined by
\[
\tilde L v=a^{ij}D_{ij}v +  (b^i+2a^{ij}D_j \zeta/\zeta) D_i v.
\]
Let $v=\underline H_f$ be the Perron's solution for the problem
\[
\tilde Lv=0 \;\text{ in }\;\Omega, \quad v=f/\zeta \;\text{ on }\;\partial\Omega.
\]
Since $\zeta \in C(\overline{\mathcal{B}})$, we have $f/\zeta \in C(\partial\Omega)$.
By Theorem \ref{cor0800sat}, we deduce that $v \in C(\overline{\Omega})$ and $v = f/\zeta$ on $\partial\Omega$.
Furthermore, from Lemmas \ref{lem1210thu} and \ref{lem_solution}, we conclude that $v \in W^{2,p_0}_{\rm loc}(\Omega)$.
Given the properties of $\zeta$, it is evident that $\zeta v \in W^{2, p_0/2}_{\rm loc}(\Omega) \cap C(\overline{\Omega})$.
Additionally, as discussed after Proposition \ref{lem1012thu}, $u=\zeta v$ satisfies the problem \eqref{eq0751sat}.
Uniqueness of $u$ is an immediate consequence of the maximum principle (see \cite[Corollary 4.3]{Krylov21}).
\end{proof}

\section{Green's function in regular domains}			\label{sec6}

In this section, we construct the Green's functions for $L$ in regular domains $\Omega \Subset \mathcal{B}''$.
Let $G_{\mathcal{B}}(x,y)$ denote the Green's function in $\mathcal{B}$, as constructed in Section \ref{sec3}.
To construct $G_\Omega(x,y)$, the Green's function in $\Omega$, we follow these steps.
For each $y \in \Omega$, we consider the Dirichlet problem:
\begin{equation}				\label{eq1645sun}
\left\{
\begin{array}{ccc}
Lu=0 &\text{ in } & \Omega,\\
u=-G_{\mathcal{B}}(\,\cdot, y)& \text{ on }&\partial \Omega.
\end{array}
\right.
\end{equation}
Given that $G_{\mathcal{B}}(\,\cdot, y)$ is continuous on $\partial\Omega$ and $\Omega$ is a regular domain, by Theorem \ref{thm0802sat}, a unique solution for the problem exists.
We denote it  by $K(\,\cdot, y)$.
It is evident from the construction that  $K(\,\cdot,y)$ is continuous in $\overline\Omega$ for each fixed $y$.
We aim to demonstrate that, for each fixed $x \in \Omega$, the mapping $K(x, \cdot\,)$ is also continuous in $\Omega$.
Indeed, since $\Omega$ is regular, for each $x \in \Omega$, we have harmonic measure $\omega^x_\Omega$ defined on $\partial \Omega$.
Utilizing the fact that $K(\,\cdot, y)$ is the solution to the problem \eqref{eq1645sun}, we can express it as:
\begin{equation}			\label{eq1049sat}
K(x,y)=-\int_{\partial\Omega} G_{\mathcal{B}}(\,\cdot,y)\, d\omega^x_\Omega.
\end{equation}
Therefore, $K(x,y)$ is continuous in $y$ for each $x \in \Omega$ as we claimed.

We then define $G_\Omega(x,y)=G_{\mathcal{B}}(x,y)+K(x,y)$.
To verify that $G_\Omega$ is the Green's function in $\Omega$, for any $f \in C^\infty_c(\Omega)$, consider the function $v$ defined by
\[
v(x)=\int_{\Omega} G_{\Omega}(x,y) f(y)\,dy=\int_{\Omega} G_{\mathcal{B}}(x,y) f(y)\,dy+\int_{\Omega} K(x,y) f(y)\,dy.
\]

As $f \in C^\infty_c(\Omega)\subset C^\infty_c(\mathcal{B})$, if we set
\[
w(x)=\int_{\Omega} G_{\mathcal{B}}(x,y) f(y)\,dy=\int_{\mathcal{B}} G_{\mathcal{B}}(x,y) f(y)\,dy,
\]
then according to \cite{HK20}, $w$ becomes the unique solution to the problem
\[
\left\{
\begin{array}{ccc}
Lw=-f &\text{ in } & \mathcal{B},\\
w=0 & \text{ on }&\partial \mathcal{B}.
\end{array}
\right.
\]
In particular, we have $Lw=-f$ in $\Omega$.
Also, we have
\[
L\int_\Omega K(\,\cdot,y) f(y)\,dy=0.
\]
Therefore, we have $Lv=-f$ in $\Omega$.
Additionally, $v=0$ on $\partial\Omega$ as $K(\,\cdot,y)=-G_{\mathcal{B}}(\,\cdot, y)$ on $\partial\Omega$.

\begin{theorem}			\label{thm1127sat}
Assuming Conditions \ref{cond1} and \ref{cond2}, let $\Omega \subset \bR^d$ be a bounded regular domain.
Then there exists the Green's function $G_{\Omega}(x,y)$ in $\Omega$, and it has the following pointwise bound:
\[
0\le G_{\Omega}(x,y) \le N \abs{x-y}^{2-d},\quad \forall x\neq y \in \Omega,
\]
where $N$ is a constant depending only on $d$, $\lambda$, $\Lambda$, $\omega_{\mathbf A}$, $p_0$, $\diam \Omega$.
\end{theorem}
\begin{proof}
We may assume that $0 \in \Omega$.
Let $\mathcal{B}=B_{4R}(0)$, $\mathcal{B}'=B_{2R}(0)$, and $\mathcal{B}''=B_R(0)$, where $R$ is chosen sufficiently large such $\Omega \Subset \mathcal{B}''$.
In the discussion above, we have shown that $G_\Omega(x,y)=G_{\mathcal{B}}(x,y)+K(x,y)$ is the Green's function in $\Omega$.
From \eqref{eq1049sat} and the fact that $G_\mathcal{B}(\cdot ,y) \ge 0$, we find that $K(x,y) \le 0$.
Therefore, we have
\[
G_\Omega(x,y) \le G_\mathcal{B}(x,y) \le N \abs{x-y}^{2-d},
\]
where $N$ is the constant from Theorem \ref{thm_green_function}.
We also note that \eqref{eq1049sat}, along with Theorem~\ref{thm_green_function}, implies that for sufficiently small $\epsilon>0$, we have $G_\Omega(x,y) \ge 0$ for every $x \in \partial B_\epsilon(y)$.
As $G_\Omega(x,y)=0$ for $x\in \partial\Omega$, we deduce that $G_\Omega(x,y) \ge 0$ by the maximum principle.
The theorem follows from the observation that $\diam \mathcal{B}=4R$ and $\dist(\partial \mathcal{B}', \mathcal{B})=2R$.
\end{proof}

\appendix
\section{}			\label{sec7}
To make the article self-contained, we provide proofs of some lemmas that are well-known in the context of abstract potential theory.

We will frequently use Baire's theorem on semicontinuous functions, which states that if $f$ is a lower semicontinuous function (l.s.c.) that does not take the value $-\infty$, there exists an increasing sequence of continuous functions $\set{f_n}$ such that $f_n$ converges to $f$ pointwise.

\subsection{Proof of Lemma~\ref{lem2334sun}}
\begin{lemma}[Lemma \ref{lem2334sun}]			
Let $u \in \mathfrak{S}^+(\mathcal{D})$. The following statements are valid:
\begin{enumerate}
\item (Strong minimum principle)
Let $\mathcal{D}$ be connected.
If $u \in \mathfrak{S}^+(\mathcal{D})$, then the infimum of $u$ is not attained in $\mathcal{D}$ unless $u$ is constant on $\mathcal{D}$.
\item (Comparison principle)
If $u \in \mathfrak{S}^+(\mathcal{D})$ and $\liminf_{y \to x,\; y\in \mathcal{D}} u(y) \ge 0$ for every $x \in \partial \mathcal{D}$, then $u \ge 0$ in $\mathcal{D}$.
\item
If $u \in \mathfrak{S}^+(\mathcal{D})$, then for every $x_0 \in \mathcal{D}$, we have $\liminf_{x \to x_0} u(x)=u(x_0)$.
\item
If $u$, $v \in \mathfrak{S}^+(\mathcal{D})$, and $c>0$, then $cu$, $u+v$, and $\min(u,v) \in \mathfrak{S}^+(\mathcal{D})$.
\item
(Pasting lemma)
For $v \in \mathfrak{S}^+(\mathcal{D}')$, where $\mathcal{D}'$ is an open subset of $\mathcal D$, define $w$ by
\[
w= \begin{cases}
\min(u,v)& \text{in }\;\mathcal{D}',\\
\phantom{\min}u &  \text{in }\; \mathcal{D} \setminus \mathcal{D}'.
\end{cases}
\]
If $w$ is lower semicontinuous in $\partial \mathcal{D}'$, then $w \in \mathfrak{S}^+(\mathcal{D})$.
\end{enumerate}
\end{lemma}

\begin{proof}
\begin{enumerate}[leftmargin=*]
\item
For any ball $B \subset \mathcal{D}$ and a function$f \in C(\partial B)$, consider
\[
v(x)=\int_{\partial B} f\,d\omega_{B}^x,\quad x \in B.
\]
We claim that if $f\ge 0$ and $f(x_0)>0$ for some $x_0 \in \partial B$, then $v>0$ in $B$.
Indeed, it follows from Lemma \ref{lem_solution} that $v \in W^{2,d}_{\rm loc}(B) \cap C(\overline B)$.
It is evident that $v \ge 0$ in $B$ and by continuity $v>0$ near $x_0$.
Then by the Harnack inequality as given in \cite[Theorem 3.1]{Safonov}, we conclude that $v >0$ in $B$ as claimed.

Suppose there is a point $x_0 \in \mathcal{D}$ such that $u(x_0)=\inf_{\mathcal{D}}u$.
Since $u>-\infty$ and $u$ is finite at some point in $\mathcal{D}$, we have $-\infty < u(x_0)=\inf_{\mathcal{D}}u<+\infty$.
Let $K=\set{x \in \mathcal{D}: u(x)=\inf_{\mathcal{D}}u}$, which is a relatively closed in $\mathcal{D}$ by lower semicontinuity of $u$.
We will show that $K$ is also open.
For any $y \in K$, there is $r>0$ such that $B_r(y) \subset \mathcal{D}$.
Suppose there is a point $y' \in B_r(y)\setminus K$. Let $\rho=\abs{y'-y}<r$.
Since $y \in K$, we have
\begin{equation}			\label{eq1113sun}
u(y)=\int_{\partial B_\rho(y)}\left(\inf_{\mathcal{D}}u\right) \,d\omega_{B_\rho(y)}^y(z) \le \int_{\partial B_\rho(y)} u(z) \,d\omega_{B_\rho(y)}^y(z) \le u(y).
\end{equation}
Consequently, we have
\[
\int_{\partial B_\rho(y)} (u(z)-u(y))\,d\omega_{B_\rho(y)}^y(z) =0.
\]
Since $u(y') > u(y)$ and $u$ is lower semicontinuous, there exists a continuous function $f$ such that $0\le f(z) \le u(z)-u(y)$ for every $z \in \partial B_\rho(y)$ and $f(y')>0$.
Define
\[
v(x)=\int_{\partial B_\rho(y)} f\,d\omega_{B_\rho(y)}^x.
\]
Then $v(y)>0$ by the observation above, which contradicts the following:
\[
v(y) =\int_{\partial B_\rho(y)} f(z)\,d\omega_{B_\rho(y)}^y(z) \le \int_{\partial B_\rho(y)} (u(z)-u(y))\,d\omega_{B_\rho(y)}^y(z)=0.
\]
Hence, $B_r(y) \subset K$ and $K$ is open.
Consequently, by the connectedness of $\mathcal{D}$, we have $K=\mathcal{D}$, as the case $K=\emptyset$ is initially excluded.
Therefore, if $u$ attains its minimum value at a point in $\mathcal{D}$, it follows that $u$ is constant.

\item
We extend $u$ to be zero in $\overline{\mathcal{B}} \setminus \mathcal{D}$.
Thus, $u$ is lower semicontinuous in $\overline{\mathcal{B}}$ and achieves its minimum value $m$ at some point $x_0\in \overline{\mathcal{B}}$.
Let
\[
U=\set{x \in \overline{\mathcal{B}}:u(x) = m}.
\]
As $x_0 \in U$, we know $U \neq \emptyset$.
We aim to demonstrate that $U$ is relatively open in $\mathcal{B}$ if $m<0$.
Take $x \in U$. Since $u(x)=m<0$, we have $x \in \mathcal{D}$.
Consider a ball $B_r(x) \Subset \mathcal{D}$.
As $u$ is an $L$-supersolution in $B_r(x)$, by the strong minimum principle, it must be constant in $B_r(x)$, implying $u=m$ in $B_r(x)$.
This establishes that the set $U$ is a nonempty open set in $\mathcal{B}$ if $m<0$.
Furthermore, let $V = \set{x \in \overline{\mathcal{B}}: u > m}$, which is open in $\mathcal{B}$ and nonempty since $0>m$.
Given $\overline{\mathcal{B}}= U\cup V$, the assumption that $m < 0$ contradicts the connectedness of $\overline{\mathcal{B}}$.
Thus, $m \ge 0$ on $\overline{\mathcal{B}}$, and consequently, $u \ge 0$ in $\mathcal{D}$.
\item
Since $u$ is lower semicontinuous, we only need to prove
\[
\liminf_{x \to x_0} u(x)\le u(x_0).
\]
Let $B_r(x_0)\Subset \mathcal{D}$.
Then for any $0<\rho<r$, we have
\[
u(x_0) \ge \int_{\partial B_\rho(x_0)} u \,d\omega^x_{B\rho(x_0)} \ge \int_{\partial B_\rho(x_0)} \left(\inf_{B_r(x_0)\setminus \set{x_0}}u\right) \,d\omega^x_{B\rho(x_0) }=\inf_{B_r(x_0)\setminus \set{x_0}} u.
\]
Since this holds as long as $B_r(x_0)\Subset \mathcal{D}$, we have
\[
u(x_0) \ge \lim_{r\to 0} \inf_{B_r(x_0)\setminus \set{x_0}} u =\liminf_{x\to x_0} u(x).	
\]
\item
It is straightforward from the definition of $L$-supersolutions.
\item
We may assume that $v \le u$ in $\mathcal{D}'$ by replacing $v$ by $\min(u,v)$, which is an $L$-supersolution in $\mathcal{D}'$ by (d).
Also, we may assume that $\mathcal{D}$ is connected.

It is clear that $w$ is not identically $+\infty$, $w >-\infty$, and is lower semicontinuous in $\mathcal{D}$.
For any point $x \in \mathcal{D}$, consider the following two cases: $x \in \mathcal{D}'$ and $x \in \mathcal{D}\setminus \mathcal{D}'$.
If $x \in \mathcal{D}\setminus \mathcal{D}'$, then for any ball $B \Subset \mathcal{D}$ with $x\in B$, we have
\[
\int_{\partial B} w\,d\omega^x_B \le \int_{\partial B} u\,d\omega^x_B \le u(x)=w(x).
\]
If $x \in \mathcal{D}'$, then consider any ball $B'$ such that $x\in B' \Subset \mathcal{D}'$, and observe that
\[
\int_{\partial B'}w \,d\omega^x_{B'}=\int_{\partial B'}v \,d\omega^x_{B'} \le v(x) = w(x).
\]
We claim that $w$ satisfies the strong minimum principle.
Indeed, we repeat the same proof as in (a).
Let $K=\set{x \in \mathcal{D}: w(x)=\inf_{\mathcal{D}}w}$.
For $y \in K$, consider two cases: if $y \in K \cap \mathcal{D}'$ choose $r>0$ sufficiently small so that $B_r(y)\Subset \mathcal{D}'$; otherwise take any $r>0$ such that $B_r(y)\Subset \mathcal{D}$.
Suppose there is a point $y' \in B_r(y)\setminus K$. Let $\rho=\abs{y'-y}<r$.
Then, by the observation above, similar to \eqref{eq1113sun}, we obtain
\[
w(y) = \int_{\partial B_\rho(y)} w(z) \,d\omega_{B_\rho(y)}^y(z).
\]
By the same reasoning as in the proof of (a), this implies that $K=\mathcal{D}$.
Thus, $w$ satisfies the strong minimum principle in $\mathcal{D}$.

Next, consider any ball $B \Subset \mathcal{D}$ and $x \in B$.
Choose a sequence $h_n \in C(\partial B)$ such that $h_n \nearrow w$ on $\partial B$.
Let $w_n(x):=\int_{\partial B} h_n \,d\omega^x_B$.
Then, $w-w_n$ satisfies the strong minimum principle in $B$ by the same reasoning.
Consequently, we have $w_n \le w$ in $B$.
Taking the limit as $n\to \infty$, we deduce that
\[
\tilde w(x):=\int_{\partial B} w \,d\omega^x_B \le w(x).
\]
This confirms that $w$ is an $L$-supersolution in $\mathcal{D}$.\qedhere
\end{enumerate}
\end{proof}

\subsection{Proof of Lemma~\ref{lem1033thu}}
\begin{lemma}[Lemma \ref{lem1033thu}]			
The following properties hold:
\begin{enumerate}
\item
$\mathscr{E}_{B}u \le u$.
\item
$\mathscr{E}_{B}u$ is an $L$-supersolution in $\mathcal{D}$.
\item
$\mathscr{E}_{B}u$ is an $L$-solution in $B$.
\end{enumerate}
\end{lemma}
\begin{proof}
We will make use of the following lemma in the proof.
\begin{lemma}			\label{lem_harnack_principle}
Assume Conditions \ref{cond1} and \ref{cond2} hold, with $c = 0$.
Let $\set{u_n}$ be an increasing sequence of  $L$-solutions in $\mathcal{D}$, and let $u := \lim_{n \to\infty} u_n$.
Then, $u \equiv +\infty$ or $u$ is an $L$-solution in $\mathcal{D}$.
\end{lemma}
\begin{proof}
We may assume that $u_n \ge 0$ by replacing $u_n$ by $u_n-u_1$.
Suppose that there exists a point $x \in \mathcal{D}$ for which $u(x)$ is finite.
Let $B_{16r} \Subset \mathcal{D}$ be such that $x \in B_{16r}$.
We use \cite[Theorem 3.1]{Safonov} to deduce that
\[
u_n(y) \le \sup_{B_{2r}} u_n \le C \inf_{B_r}u_n \le  C u_n(x) \le  C u(x),\quad \forall  y \in B_{2r},\;\;n=1,2,\ldots.
\]
Moreover, according to Lemma~\ref{lem_solution}, $u_n \in W^{2,p_0}_{\rm loc}(\mathcal{D})$, and thus, similar to \eqref{eq1057thu}, we have
\[
\norm{u_n}_{W^{2,p_0}(B_{r})} \le C \norm{u_n}_{L^{p_0}(B_{2r})} \le C u(x).
\]
This implies that $u \in W^{2,p_0}(B_{r})$ as well, and $u$ is continuous in $\overline B_r$ by the Sobolev inequality.
Since $u_n$ is an $L$-solution in $\mathcal{D}$, it follows from Lemma \ref{lem_solution} that
\[
u_n(x) = \int_{\partial B_r} u_n\,d\omega^x_{B_r}, \quad \forall  x \in B_r.
\]
By taking the limit $n\to \infty$, invoking the monotone convergence theorem, and repeating the proof of Lemma \ref{lem_solution}, we conclude that $Lu=0$ a.e. in $B_r$.
Finally, by using a chain of balls, we deduce that $u$ is an $L$-solution in $\mathcal{D}$.
\end{proof}

\begin{enumerate}[leftmargin=*]
\item
It is clear from the definition.
\item
By the pasting lemma, it suffices to show that $\mathscr{E}_{B}u$ is lower semicontinuous on $\partial B$.
Let $x_0 \in \partial B$.
From the definition of $\mathscr{E}_{B}u$, it is enough to show that
\begin{equation}			\label{eq1137wed}
u(x_0)\le \liminf_{x \to x_0,\; x \in B}\int_{\partial B} u \ d\omega^x_{B}.
\end{equation}
Choose a sequence $h_n \in C(\partial B)$ such that $h_n \nearrow u$ on $\partial B$.
Then, we have
 \[
 h_n(x_0) = \lim_{x \to x_0,\; x \in B}\int_{\partial B} h_n \ d\omega^x_{B} =\liminf_{x \to x_0,\; x \in B} \int_{\partial B} h_n \ d\omega^x_{B} \le \liminf_{x \to x_0,\; x \in B}\int_{\partial B} u \ d\omega^x_{B}.
 \]
By taking the limit $n\to\infty$, we obtain \eqref{eq1137wed}.
\item
Consider a sequence $h_n \in C(\partial B)$ such that $h_n \nearrow u$ on $\partial B$.
Define
\[
v_n(x)=\int_{\partial B} h_n \ d\omega^x_{B}\quad\text{and}\quad v(x)=\int_{\partial B} u \ d\omega^x_{B}.
\]
By Lemma~\ref{lem_solution} and \cite[Theorem 1.1]{Safonov}, it follows that  $v_n \in W^{2,p_0}_{\rm loc}(B)\cap C(\overline B)$ forms an increasing sequence of $L$-solutions in $B$.
As $v_1(x)$ is finite for $x \in B$, the monotone convergence theorem implies that $v(x)=\lim_{n \to \infty} v_n(x)$ for $x \in B$.
Since $u$ is an $L$-supersolution, we have $v \leq u$ in $B$, thus $v$ is not identically equal to $+\infty$ in $B$. Therefore, by Lemma \ref{lem_harnack_principle}, we have $L v=0$ in $B$.\qedhere
\end{enumerate}
\end{proof}

\subsection{Proof of Lemma \ref{lem1034thu}}
\begin{lemma}[Lemma \ref{lem1034thu}]			
Let $\mathscr{F}$ be a saturated collection of $L$-supersolutions in $\mathcal{D}$, and let
\[
v(x) := \inf_{u \in \mathscr{F}} u(x).
\]
If $v$ does not take the value $-\infty$ in $\mathcal{D}$, then $Lv=0$ in $\mathcal{D}$.
\end{lemma}

\begin{proof}
We recall  a lemma by Choquet (see, for example, \cite[p. 169]{CC72}) that asserts the existence of a sequence of functions $\set{u_n}$ in $\mathscr{F}$ such that $u := \inf_{n} u_n$ satisfies $\hat{u}=\hat{v}$, where $\hat{u}$ and $\hat{v}$ are the lower semicontinuous regularizations of $u$ and $v$, respectively, as defined in Definition \ref{def2249sat}.

Suppose $v>-\infty$ in $\mathcal{D}$.
Then, the following inequalities holds in $B \Subset \mathcal{D}$:
\begin{equation}			\label{eq1103sun}
\hat{v}\leq v \leq  \mathscr{E}_{B}u_n \le u_n.
\end{equation}
Since $\set{\mathscr{E}_{B}u_n}$ is a decreasing sequence, by Lemma \ref{lem_harnack_principle}, its limit
\[
\tilde{u} := \lim_{n \to \infty} \mathscr{E}_{B}u_n
\]
is an $L$-solution in $B$.
By taking the limit in \eqref{eq1103sun} and then taking the lower semicontinuous regularization, we deduce that $\hat{v} \le v \le \tilde{u} \le \hat{u}$ in $B$.
Therefore, we conclude that $v=\tilde u$ in $B$, and thus we have $Lv=0$ in $B$.
Since $B \Subset \mathcal{D}$ is arbitrary, the proof is completed.
\end{proof}

\subsection{Proof of Lemma~\ref{lem1655sun}}
\begin{lemma}[Lemma \ref{lem1655sun}]			
The followings are true:
\begin{enumerate}[leftmargin=*]
\item
$\hat u \le u$.
\item
$\hat u$ is lower semicontinuous .
\item
If $v$ is a lower semicontinuous function satisfying $v \le u$, then $v \le \hat u$.
\item
If $u$ is lower semicontinuous, then $\hat u=u$.
\end{enumerate}
\end{lemma}

\begin{proof}
\begin{enumerate}[leftmargin=*]
\item
It is clear from the definition.
\item
Let $r>0$. For any $y \in \mathcal{D} \cap B_r(x)\setminus \set{x}$, we have
\[
\inf_{\mathcal D \cap B_{2r}(x)} u  \le \inf_{\mathcal D \cap B_r(y)} u \le \hat u(y).
\]
By taking the infimum over $y \in \mathcal{D} \cap B_r(x)\setminus \set{x}$, we deduce that
\[
\inf_{\mathcal D \cap B_{2r}(x)} u \le \inf_{\mathcal D \cap B_r(x)\setminus \set{x}} \hat u(y).
\]
By taking the supremum over $r>0$, we conclude that $\hat u(x) \le \liminf_{y \to x} \hat u(y)$.
\item
Let $v$ be a lower semicontinuous function satisfying $v \le u$ in $\mathcal{D}$.
It is clear that $\hat v \le \hat u$.
From the definition of lower semicontinuity, we have
\[
v(x) \le \min \left(v(x),\, \sup_{r>0} \inf_{B_r(x) \setminus \set{x}} v\right)=\sup_{r>0}\inf_{B_r(x)} v=\hat v(x).
\]
\item
It follows from (a) and (c).\qedhere
\end{enumerate}
\end{proof}

\subsection{Proof of Lemma~\ref{lem1013sat}}
\begin{lemma}[Lemma \ref{lem1013sat}]				
The following properties hold:
\begin{enumerate}[leftmargin=*]
\item
$0\le \hat u_E \le u_E \le 1$.
\item
$u_E=1$ in $E$ and $\hat u_E =1$ in $\intr (E)$.
\item
$\hat u_E$ is a potential.
\item
$u_E=\hat u_E$ in $\mathcal{B}\setminus \overline{E}$ and
$L u_E=L \hat u_E=0$ in $\mathcal{B}\setminus \overline{E}$.
\item
If $x_0\in \partial E$, we have
\[
\liminf_{x\to x_0} \hat u_{E}(x) =\liminf_{x\to x_0,\; x\in \mathcal{B}\setminus  E} \hat u_E(x).
\]
\end{enumerate}
\end{lemma}
\begin{proof}
Let us denote
\[
\mathscr{F}= \set{ v \in \mathfrak{S}^+(\mathcal{B}):  v\ge 0\;\text{ in }\;\mathcal{B},\;\; v \ge 1 \;\text{ in }\;E},
\]
so that $u_E(x)=\inf_{v \in \mathscr{F}} v(x)$.
We aim to show that for  $B \Subset \mathcal{B}$ and $x\in B$, we have
\begin{equation}				\label{eq2159sat}
\int_{\partial B} u_E \,d\omega_{B}^x \le u_E(x).
\end{equation}
Indeed, for any $v \in \mathscr{F}$, we have
\[
\int_{\partial B} u_E \,d\omega_{B}^x \le \int_{\partial B} v\,d\omega_{B}^x \le v(x).
\]
Taking the infimum over $v \in \mathscr{F}$ yields \eqref{eq2159sat}.
\begin{enumerate}[leftmargin=*]
\item
Given that $1 \in \mathscr{F}$, it follows that $u_E \leq 1$.
Therefore, it is clear that
\[
0 \le \hat u_E \leq u_E \le 1.
\]
\item
For every $v \in \mathscr{F}$, $\min(v,1) \in \mathscr{F}$.
Since $\min(v,1) = 1$ on $E$, it implies that $u_E(x) =1$ for all $x \in E$.
Consequently, based on Definition \ref{def2249sat}, $\hat{u}_E(x)=1$ if $x \in \intr(E)$.
\item
We begin by demonstrating that $\hat u_E \in \mathfrak{S}^+(\mathcal{B})$.
Let $h_n \in C(\partial B)$ be a sequence of nonnegative functions such that $h_n \nearrow \hat u_E$ on $\partial B$.
Then, for each $n$, we have
\[
u_n(x):=\int_{\partial B} h_n \,d\omega_{B}^x \le \int_{\partial B} \hat u_E \,d\omega_{B}^x \le \int_{\partial B} u_E \,d\omega_{B}^x \le u_E(x),
\]
where the inequality \eqref{eq2159sat} is used.
Since $u_n$ is continuous on $\overline B$, we have $\hat u_n(x)=u_n(x)$ for $x \in B$.
Thus, we deduce
\[
\int_{\partial B} h_n \,d\omega_{B}^x=u_n(x) \le \hat u_E(x).
\]
Taking the limit in the above and applying the monotone convergence theorem, we obtain
\[
\int_{\partial B} \hat u_E \,d\omega_{B}^x \le \hat u_E(x),
\]
confirming that $\hat u_E$ satisfies the property (iv) in Definition~\ref{def_supersol}.
Since $\hat u_E$ is lower semicontinuous, and $0\le \hat u_E \le 1$, we conclude that $\hat u_E \in \mathfrak{S}^+(\mathcal{B})$.

Next, we demonstrate that $\hat{u}_E$ vanishes on $\partial \mathcal{B}$.
Let $y_0$ be a fixed point in $\mathcal{B}'$.
It is clear that $G(\,\cdot, y_0) \in \mathfrak{S}^+(\mathcal{B})$.
According to Theorem \ref{thm_green_function}, for a sufficiently large constant $k$, we have $kG(\,\cdot, y_0) \in \mathscr{F}$.
Since $G(\,\cdot, y_0)$ vanishes on $\partial \mathcal{B}$, it follows that $\hat{u}_E$ also vanishes on $\partial \mathcal{B}$.

\item
Note that $\mathscr{F}$ is saturated in $\mathcal{B} \setminus \overline{E}$.
By Lemma \ref{lem1034thu}, $u_E$ is an $L$-solution in $\mathcal{B}\setminus \overline{E}$.
Therefore, $u_E$ is continuous in $\mathcal{B}\setminus \overline{E}$.
This, along with Definition \ref{def2249sat}, implies that $\hat u_E = u_E$ in $\mathcal{B}\setminus \overline{E}$.

\item
It is enough to show that
\begin{equation}			\label{eq2053mon}
\liminf_{x\to x_0} \hat u_{E}(x) \ge \liminf_{x\to x_0,\; x\in \mathcal{B}\setminus  E} \hat u_E(x).
\end{equation}
For any $r>0$ such that $B_r(x_0) \Subset \mathcal{B}$, we have
\[
\inf_{x\in B_r(x_0)\setminus \set{x_0}} \hat u_E(x)= \inf_{x \in B_r(x_0)\setminus \set{x_0}} \left(\sup_{s > 0} \inf_{y\in B_s(x)} u_{E}(y) \right)\ge \inf_{x \in B_r(x_0)\setminus \set{x_0}} \left(\inf_{\abs{y-x}<s(x)} u_E(y)\right),
\]
where $s(x)=\min(\abs{x-x_0}, r-\abs{x-x_0})$.
Therefore, we have
\[
\inf_{x\in B_r(x_0)\setminus \set{x_0}} \hat u_E(x) \ge \inf_{x\in B_r(x_0)\setminus \set{x_0}} u_E(x).
\]
Since $\hat u_E \le u_E$ and $u_E=1$ on $E$, we have
\[
\inf_{x\in B_r(x_0)\setminus \set{x_0}} \hat u_E(x) = \inf_{x\in B_r(x_0)\setminus \set{x_0}} u_E(x)=\inf_{x\in (B_r(x_0)\setminus \set{x_0})\setminus E} u_E(x) \ge \inf_{x\in (B_r(x_0)\setminus \set{x_0})\setminus E} \hat u_E(x).
\]
By taking the limit $r\to 0$, we obtain \eqref{eq2053mon}. \qedhere
\end{enumerate}
\end{proof}

\subsection{Proof of Lemma~\ref{prop0800sun}}

\begin{definition}
We define an extended real-valued function $u$ as a nearly $L$-supersolution in $\mathcal{D}$ if it satisfies the following conditions:
\begin{enumerate}[leftmargin=*]
\item[(i)]
$u$ is not identically $+\infty$ in any connected component of $\mathcal{D}$.
\item[(ii)]
$u$ is locally uniformly bounded from below in $\mathcal{D}$.
\item[(ii)]
For any $B\Subset \mathcal{D}$ and $x \in B$, we have
\[
u(x) \ge \overline{\int_{\partial B}} u\,d\omega_{B}^x:=\inf \Set{\int_{\partial B} f\,d\omega_{B}^x:  f \text{ is l.s.c. on}\,\partial B,\; f\ge u\text{ on }\partial B}.
\]
\end{enumerate}
\end{definition}
Note that a nearly $L$-supersolution is not required to be measurable.
In the case when $u$ is an $L$-supersolution in $\mathcal{D}$, it is evident that
\[
\overline{\int_{\partial B}} u\,d\omega_{B}^x=\int_{\partial B} u\,d\omega_{B}^x,
\]
and thus, we find that $u$ is nearly $L$-supersolution in $\mathcal{D}$.
The following lemma shows the converse is also true if and $u$ is lower semicontinuous.

\begin{lemma}			\label{lem1020sun}
The followings are true:
\begin{enumerate}[leftmargin=*]
\item
If $u$ is a nearly $L$-supersolution in $\mathcal{D}$, then $\hat u$ is an $L$-supersolution in $\mathcal{D}$.
\item
Let $\mathscr F$ be a family of $L$-supersolutions in $\mathcal{D}$, and define
\[
u(x):= \inf_{v \in \mathscr F} v(x).
\]
If $u>-\infty$ in $\mathcal{D}$, then $u$ is a nearly $L$-supersolution in $\mathcal D$.
\item
Suppose $u$, $v$, and $w$ are nearly $L$-supersolution such that $u=v+w$.
Then, we have $\hat u=\hat v+ \hat w$.
\end{enumerate}
\end{lemma}
\begin{proof}
\begin{enumerate}[leftmargin=*]
\item
For any ball $B\Subset \mathcal{D}$ and any l.s.c. function $f$ such that $f \ge u$ on $\partial B$, we have $f \ge u \ge \hat u$ on $\partial B$, and thus, for any $x \in B$, we have
\[
\int_{\partial B} \hat u \, d\omega_{B}^x \le \int_{\partial B} f \, d\omega_{B}^x.
\]
Thus, by taking the infimum of all such $f$, and using the assumption that $u$ is nearly $L$-supersolution, we obtain
\[
\tilde u(x):=\int_{\partial B} \hat u\, d\omega_{B}^x \le \overline{\int_{\partial B}} u \, d\omega_{B}^x \le u(x).
\]
We will show that $\tilde u$ is continuous in $B$.
Then, we deduce from Lemma \ref{lem1655sun} that $\tilde u \le  \hat u$ in $B$.
In particular, we have for any $x\in B$,
\[
\int_{\partial B} \hat u\, d\omega_{B}^x =\tilde u(x)\le \hat u(x).
\]
It remains to show that $\tilde u$ is continuous in $B$.
Let $h_n$ be a sequence of continuous functions such that $h_n \nearrow \hat u$ on $\partial B$.
Then, we have
\[
u_n(x):=\int_{\partial B} h_n d\omega_{B}^x \nearrow \int_{\partial B} \hat u\, d\omega_{B}^x =\tilde u(x).
\]
Since $\set{u_n}$ is an increasing sequence of $L$-solutions in $B$, it follows from Lemma \ref{lem_harnack_principle} that $\tilde u$ is an $L$-solution in $B$, and thus it is continuous in $B$.

\item
For any ball $B\Subset \mathcal{D}$, $x \in B$, and $v\in \mathscr F$, we have
\[
\overline{\int_{\partial B}} u\,d\omega^x_{B} \le \overline{\int_{\partial B}} v\,d\omega^x_{B} = \int_{\partial B} v\,d\omega^x_{B} \le v(x).
\]
By taking infimum over $v \in \mathscr F$, we conclude that
\[
\overline{\int_{\partial B}} u\,d\omega^x_{B}  \le u(x).
\]
Since $u > -\infty$ and cannot be identically $+\infty$ on any connected component of $\mathcal D$, it is a nearly $L$-subsolution.

\item
We will show that for any ball $B\Subset \mathcal D$ and $x \in B$, the following holds:
\begin{equation}			\label{eq1657sun}
\overline{\int_{\partial B}} (u_1+u_2)\,d\omega^x_{B} \le \overline{\int_{\partial B}} u_1\,d\omega^x_{B} + \overline{\int_{\partial B}} u_2\,d\omega^x_{B}.
\end{equation}
Indeed, if $f_1$ and $f_2$ are l.s.c. functions such that $u_1 \le f_1$ and $u_2 \le f_2$ on $\partial B$, then $f=f_1+f_2$ is an l.s.c. function such that $u_1+u_2 \le f_1+f_2$ on $\partial B$, and thus, we have
\[
\overline{\int_{\partial B}} (u_1+u_2)\,d\omega^x_{B} \le \int_{\partial B} (f_1+f_2)\,d\omega^x_{B} = \int_{\partial B} f_1\,d\omega^x_{B} + \int_{\partial B} f_2\,d\omega^x_{B}.
\]
By taking the infimum over $f_1$ and $f_2$ separately, we obtain \eqref{eq1657sun}.

Next, we will show that if $u$ is a nearly $L$-supersolution, then we have
\begin{equation}			\label{eq1658sun}
\hat u(x)=\sup\Set{\overline{\int_{\partial B}} u \,d\omega^x_B: B \Subset \mathcal{D},\; x \in B},\quad \forall x \in \mathcal{D}.
\end{equation}
To see this, we set
\[
\tilde u(x)=\sup\Set{\overline{\int_{\partial B}} u \,d\omega^x_B: B \Subset \mathcal{D},\; x \in B}.
\]
Since $u$ is nearly $L$-supersolution, it is clear that $\tilde u(x) \le u(x)$.
We will show that $\tilde u$ is lower semicontinuous in $\mathcal{D}$.
Then, Lemma \ref{lem1655sun} implies that $\tilde u \le \hat u$ in $\mathcal{D}$, and thus, $\tilde u(x) \le \hat u(x)$.
Suppose there exists $x_0 \in \mathcal{D}$ such that
\[
\alpha:=\liminf_{y \to x_0} \tilde u(y) < \tilde u(x_0).
\]
For any number $\beta$ satisfying $\alpha<\beta<\tilde u(x_0)$, there exists a ball $B$ such that $x_0 \in B \Subset \mathcal D$ and $u_B(x_0)>\beta$, where we set
\[
u_B(y):=\overline{\int_{\partial B}} u \,d\omega^{y}_B.
\]

We claim that $u_B$ is an $L$-solution in $B$.
First, observe that a function $H_f$ defined by $H_f(y)=\int_{\partial B} f\, d\omega_B^y$ for an l.s.c. function $f$ on $\partial B$ is an $L$-solution in $B$, as detailed in the proof of Lemma \ref{lem1033thu}(c). 
Next, note that $u_B(x)=\inf_{v\in \mathscr{F}} v(x)$, where $\mathscr{F}= \set{H_f: f \text{ is l.s.c. on }\partial B,\; f\ge u\text{ on }\partial B}$.
A slight modification of the proof of Lemma \ref{lem_harnack_principle} shows that $u_B$ is an $L$-solution in $B$.
This particularly implies that $u_B$ is continuous in $B$.
Therefore, there exists another ball $B'$ such that $x_0 \in B' \Subset B$ and $u_B(y) > \beta$ for every $y \in B'$.
Consequently, we have $u(y) \ge u_B(y)>\beta$ for every $y \in B'$ as $u$ is nearly $L$-supersolution.
Therefore, we deduce that
\[
\tilde u(y) \ge \overline{\int_{\partial B'}} u \,d\omega_{B'}^y \ge \overline{\int_{\partial B'}} \beta \,d\omega_{B'}^y= \int_{\partial B'} \beta \,d\omega_{B'}^y=\beta,\quad \forall y \in B'.
\]
This implies that $\alpha=\liminf_{y\to x_0} \tilde u(y) \ge \beta$, which is a contradiction.
We have confirmed that $\tilde u(x) \le \hat u(x)$.

The see the converse inequality holds, consider a ball $B=B_r(x) \Subset \mathcal D$, and observe that
\[
\inf_{B_r(x)} u \le \int_{\partial B} \left(\inf_{\overline B} u \right) \,d\omega^x_B =\overline{\int_{\partial B}}\left(\inf_{\overline B} u \right) \,d\omega^x_B  \le \overline{\int_{\partial B}} u \,d\omega^x_B \le \tilde{u}(x).
\]
By taking the limit $r \to 0$ in the inequality above, we obtain $\hat u(x) \le \tilde u (x)$.
We have confirmed the identity \eqref{eq1658sun}.

Finally, we apply we combine \eqref{eq1657sun} and \eqref{eq1658sun} together, and deduce from $u=v+w$ that $\hat u \le \hat v+ \hat w$.
It is clear from Definition \ref{def2249sat} that $\hat u \ge \hat v + \hat w$.
This completes the proof. \qedhere
\end{enumerate}
\end{proof}

\begin{lemma}[Lemma \ref{prop0800sun}]			
Let $u$ be a nonnegative $L$-supersolution in $\mathcal{B}$, and assume $u<+\infty$ in $\mathcal B$.
For an open set $\mathcal{D}\subset \mathcal{B}$, we define
\[
\mathscr{F}_{\mathcal D}=\set{v \in \mathfrak{S}^+(\mathcal{B}):  v\ge 0\,\text{ in }\,\mathcal{B},\; v - u \in \mathfrak{S}^+(\mathcal{D})}.
\]
For $x \in \mathcal{B}$, we then define
\[
\mathbb{P}_\mathcal{D} u(x)=\inf_{v \in \mathscr{F}_{\mathcal D}} v(x).
\]
The following properties hold:
\begin{enumerate}[leftmargin=*]
\item
$\mathbb{P}_\mathcal{D} u \in \mathfrak{S}^+(\mathcal{B})$ and $0 \le \mathbb{P}_\mathcal{D} u \le u$ in $\mathcal{B}$.
\item
$\mathbb{P}_\mathcal{D} u \in \mathscr{F}_{\mathcal D}$, i.e.,  $\mathbb{P}_\mathcal{D} u - u$ is an $L$-supersolution in $\mathcal{D}$.
\item
$\mathbb{P}_\mathcal{D} u$ is an $L$-solution in $\mathcal{B}\setminus \overline{\mathcal{D}}$.
\item
If $v \in \mathscr{F}_{\mathcal{D}}$, then $v-\mathbb{P}_\mathcal{D} u$ is a nonnegative $L$-supersolution in $\mathcal{B}$.
\end{enumerate}
\end{lemma}
\begin{proof}
We define $u_1:=\mathbb{P}_\mathcal{D}u$ and $u_2:=u_1-u$.
It is clear that $u_1 \ge 0$.
Since $u \in \mathscr{F}_{\mathcal{D}}$, we have $u_1 \le u$.
Therefore, we have  $0\le u_1 \le u$.
Note that $\hat u_1 \in \mathfrak{S}^+(\mathcal{B})$ according to Lemma \ref{lem1020sun}.
Since we have
\[
u_2(x)=\mathbb{P}_\mathcal{D} u(x) - u(x)=\inf_{v\in \mathscr{F}_\mathcal{D}} (v-u)(x),
\]
and $v-u\in \mathfrak{S}^+(\mathcal{D})$ for $v \in \mathscr{F}_\mathcal{D}$,
it follows again from Lemma \ref{lem1020sun} that $\hat u_2 \in \mathfrak{S}^+(\mathcal{D})$.
Moreover, it follows from Lemma \ref{lem1020sun} that $\hat u_1 = \hat u+ \hat u_2$ in $\mathcal{D}$.
Therefore, as $\hat u =u$, we conclude that
\[
\hat u_1(x)=u(x)+\hat u_2(x),\quad \forall x \in \mathcal{D}.
\]

We have shown that $\hat u_1 \in \mathscr{F}_\mathcal{D}$.
As a result, $\hat u_1 \ge \mathbb{P}_\mathcal{D}u = u_1$, implying $u_1=\hat u_1$.
These observations confirm (a) and (b).

As it is clear that $\mathscr F_{\mathcal{D}}$ is saturated in $\mathcal{B} \setminus \overline{\mathcal{D}}$, Lemma \ref{lem1034thu} establishes (c).

Now, it only remains to establish (d).
Let $v \in \mathscr{F}_{\mathcal D}$.
For any $B\Subset \mathcal{B}$, we define
\[
w(x)= \int_{\partial B} f\,d\omega_B^x\quad\text{for }\; x \in B,
\]
where $f \in C(\partial B)$ is any function satisfying $f \le (v-u_1)$ on $\partial B$.
Note that $w \in C(\overline B)$.
If we show that $u_1 \le v-w$ in $B$, then this will imply that $v-u_1$ is a nearly $L$-supersolution in $\mathcal{B}$.
Let us define $v_1$ by
\[
v_1= \begin{cases}
\min(v-w,u_1)& \text{in }\;B,\\
\phantom{\min}u_1 &  \text{in }\; \mathcal{B} \setminus B.
\end{cases}
\]
Then, for $x_0 \in \partial B$, due to the lower semicontinuity of $u_1$ and $v$ at $x_0$, along with the continuity of $w$ at $x_0$, we have
\begin{align*}
v_1(x_0) =u_1(x_0) &\le \min\left(\liminf_{x\to x_0} u_1(x), v(x_0)-w(x_0)\right)\\
&\le \min\left(\liminf_{x\to x_0} u_1(x),\, \liminf_{x\to x_0} v(x) -w(x_0)\right)\\
& \le \min\left(\liminf_{x\to x_0} u_1(x),\, \liminf_{x\to x_0,\,x\in B}  \,(v -w)(x)\right)\le \liminf_{x\to x_0} v_1(x).
\end{align*}
We have established the lower semicontinuity of $v_1$ on $\partial B$.
Utilizing the pasting lemma, we infer that $v_1 \in \mathfrak{S}^+(\mathcal B)$.
As $w$ is an $L$-solution in $B$ and $w\le v-u_1\le v$ on $\partial B$, the comparison principle implies $w \le v$ in $B$.
Consequently, from the definition of $v_1$, we deduce that $v_1 \ge 0$ in $\mathcal{B}$.
Thus, we have $v_1 \in \mathscr{F}_\mathcal{D}$ once we establish $v_1-u$ is an $L$-supersolution in $\mathcal D$.
To see this, let us consider $v_1-u$ restricted on $\mathcal{D}$.
It follows from the definition of $v_1$ that
\[
v_1-u= \begin{cases}
\min(v-u-w,u_1-u)& \text{in }\;B\cap \mathcal{D},\\
\phantom{\min}u_1-u &  \text{in }\; \mathcal{D} \setminus B.
\end{cases}
\]
Since $u_1$, $v\in \mathscr{F}_{\mathcal{D}}$ and $w$ is an $L$-solution in $B$, according to the pasting lemma, $v_1-u \in \mathfrak{S}^+(\mathcal D)$ provided that $v_1-u$ is lower semicontinuous in $\mathcal{D}\cap \partial B$.
For $x_0 \in \mathcal{D}\cap \partial B$, by utilizing the semicontinuity of $v-u$ and continuity of $w$ at $x_0$, along with the inequality $w(x_0)\le v(x_0)-u_1(x_0)$, we deduce
\[
\liminf_{x\to x_0,\; x \in \mathcal{D}\cap \partial B} (v-u-w)(x) \ge (v-u)(x_0)-w(x_0) \ge u_1(x_0)-u(x_0)=v_1(x_0)-u(x_0).
\]

Upon confirming that $v_1 \in \mathscr{F}_{\mathcal{D}}$, we establish $v_1 \ge \mathbb{P}_\mathcal{D} u=u_1$ in $\mathcal{B}$.
Considering the definition, it is evident that $v_1 \le u_1$ in $\mathcal{B}$.
Consequently, we deduce that $v_1=u_1$ in $\mathcal{B}$, implying $\min(v-w,u_1)=u_1$ in $B$.
Therefore, we conclude that $v-w \le u_1$ in $B$, and this confirms that $v-u_1$ is a nearly $L$-supersolution in $\mathcal{B}$.

Writing $v_2=v-u_1$ so that $v=u_1+v_2$, and applying Lemma \ref{lem1020sun}, we conclude that $v=\hat v=\hat u_1+\hat v_2=u_1+\hat v_2$.
From the assumption that $v \in \mathscr{F}_{\mathcal{D}}$, it follows $v_2=v-u_1=v-\mathbb{P}_\mathcal{D} u \ge 0$.
As it is clear that $v_2$ is not identically $+\infty$, we conclude that $\hat v_2$ is a nonnegative $L$-supersolution in $\mathcal{B}$. \qedhere
\end{proof}

\textbf{Conflict of interest statement }
On behalf of all authors, the corresponding author states that there is no conflict of interest.

\textbf{Data availability statement }
No datasets were generated or analyzed during the current study.


\end{document}